\documentclass[leqno]{amsart}
\usepackage{times}
\usepackage{amsfonts,amssymb,amsmath,amsgen,amsthm}
\usepackage{hyperref}
\usepackage{color}
\newcommand{\msc}[2][2000]{%
  \let\@oldtitle\@title%
  \gdef\@title{\@oldtitle\footnotetext{#1 \emph{Mathematics subject
        classification.} #2}}%
}

\theoremstyle{plain}
\newtheorem{theorem}{Theorem}[section]
\newtheorem{definition}[theorem]{Definition}
\newtheorem{assumption}[theorem]{Assumption}
\newtheorem{lemma}[theorem]{Lemma}
\newtheorem{corollary}[theorem]{Corollary}
\newtheorem{proposition}[theorem]{Proposition}
\newtheorem{hyp}[theorem]{Assumption}
\theoremstyle{remark}
\newtheorem{remark}[theorem]{Remark}

\def\C{{\mathbb C}}
\def\R{{\mathbb R}}
\def\N{{\mathbb N}}
\def\Sch{{\mathcal S}}
\def\O{\mathcal O}

\def\({\left(}
\def\){\right)}
\def\<{\left\langle}
\def\>{\right\rangle}
\def\le{\leqslant}
\def\ge{\geqslant}

\def\Eq#1#2{\mathop{\sim}\limits_{#1\rightarrow#2}}
\def\Tend#1#2{\mathop{\longrightarrow}\limits_{#1\rightarrow#2}}

\def\d{{\partial}}
\def\eps{\varepsilon}
\def\l{\lambda}

\def\si{{\sigma}}

\DeclareMathOperator{\RE}{Re}
\DeclareMathOperator{\IM}{Im}

\numberwithin{equation}{section}

\begin{document}

\title[Semi-classical nonlinear quantum scattering]{On semi-classical
  limit of nonlinear quantum scattering}

\author[R. Carles]{R\'emi Carles}
\address{CNRS \& Univ. Montpellier\\Math\'ematiques
\\CC~051\\34095 Montpellier\\ France}
\email{Remi.Carles@math.cnrs.fr}

\begin{abstract}
We consider the nonlinear Schr\"odinger equation with a short-range external
potential, in a semi-classical scaling. We show that for fixed Planck
constant, a complete scattering theory is available, showing that both
the potential and the nonlinearity are asymptotically negligible for
large time. Then, for data under the form of coherent
state, we show that a scattering 
theory is also available for the approximate envelope of the
propagated coherent state, which is given by a nonlinear
equation. In the semi-classical limit, these two scattering
operators can be compared in terms of classical scattering theory,
thanks to a uniform in time error estimate. Finally, we infer a large
time decoupling phenomenon in the case of finitely many initial coherent
states. 
\end{abstract}
\thanks{This work was supported by the French ANR projects
  SchEq (ANR-12-JS01-0005-01) and BECASIM
  (ANR-12-MONU-0007-04).} 
\maketitle

\section{Introduction}
\label{sec:intro}

We consider the equation
\begin{equation}
  \label{eq:1}
  i\eps\d_t \psi^\eps +\frac{\eps^2}{2}\Delta \psi^\eps =
  V(x)\psi^\eps + |\psi^\eps|^2 \psi^\eps,\quad (t,x)\in \R\times \R^3,
\end{equation}
and both semi-classical ($\eps\to 0$) and large time ($t\to \pm
\infty$) limits. Of course these limits must not be expected to
commute, and one of the goals of this paper is to analyze this lack of
commutation on specific asymptotic data, under the form of coherent
states as described below. Even though our main  result
(Theorem~\ref{theo:cv}) is proven specifically for the above case of a
cubic three-dimensional equation, two important intermediate results
(Theorems~\ref{theo:scatt-quant} and \ref{theo:scatt-class}) are
established in a more general setting. Unless specified otherwise, we
shall from now on consider $\psi^\eps:\R_t \times \R^d_x\to \C$, $d\ge
1$.

\subsection{Propagation of initial coherent states}
\label{sec:prop-init-coher}

In this subsection, we consider the initial value problem, as opposed
to the scattering problem treated throughout this paper. More
precisely, we assume here that the wave function is, at time $t=0$,
given by the coherent state
\begin{equation}
  \label{eq:ci}
  \psi^\eps(0,x) = \frac{1}{\eps^{d/4}}a\(\frac{x-q_0}{\sqrt\eps}\)
  e^{ip_0\cdot (x-q_0)/\eps},
\end{equation}
where $q_0,p_0\in \R^d$ denote the initial position and velocity,
respectively. The function $a$ belongs to the Schwartz class,
typically. In the case where $a$ is a (complex) Gaussian, many
explicit computations are available in the linear case (see
\cite{Hag80}). Note that the $L^2$-norm of $\psi^\eps$ is independent
of $\eps$, $\|\psi^\eps(t,\cdot)\|_{L^2(\R^d)} =\|a\|_{L^2(\R^d)}$. 

 Throughout this subsection, we assume that the external
potential $V$ is smooth and real-valued, $V\in C^\infty(\R^d;\R)$, and
at most quadratic, in the sense that 
\begin{equation*}
  \d^\alpha V\in L^\infty(\R^d),\quad \forall |\alpha|\ge 2.
\end{equation*}
This assumption will be strengthened when large time behavior is
analyzed. 
\subsubsection{Linear case}
\label{sec:linear-case}
 Resume \eqref{eq:1} in the absence of nonlinear term:
\begin{equation}
  \label{eq:lin}
  i\eps\d_t \psi^\eps +\frac{\eps^2}{2}\Delta \psi^\eps =
  V(x)\psi^\eps,\quad x\in \R^d,
\end{equation}
associated with the initial datum \eqref{eq:ci}. To derive an
approximate solution, and to describe the propagation of the initial
wave packet, introduce the Hamiltonian flow
\begin{equation}
  \label{eq:hamil}
  \dot q(t)= p(t),\quad \dot p(t)=-\nabla V\(q(t)\),
\end{equation}
and prescribe the initial data $q(0)=q_0$, $p(0)=p_0$. Since the
potential $V$ is smooth and at most quadratic, the solution
$(q(t),p(t))$ is smooth, defined for all time, and grows at most
exponentially.
The classical action is given by
\begin{equation}\label{eq:action}
S(t)=\int_0^t \left( \frac{1}{2} |p(s)|^2-V(q(s))\right)\,ds.
\end{equation}
We observe that if we change the unknown function $\psi^\eps$ to
$u^\eps$ by
\begin{equation}
  \label{eq:chginc}
  \psi^\eps(t,x)=\eps^{-d/4} u^\eps 
\left(t,\frac{x-q(t)}{\sqrt\eps}\right)e^{i\left(S(t)+p(t)\cdot
    (x-q(t))\right)/\eps},
\end{equation}
then, in terms of $u^\eps=u^\eps(t,y)$, the Cauchy problem
\eqref{eq:lin}--\eqref{eq:ci} is equivalent to 
\begin{equation}\label{eq:ueps0}
i\d_t
u^\eps+\frac{1}{2}\Delta u^\eps=V^\eps(t,y)
u^\eps\quad ;\quad u^\eps(0,y) = a(y),
\end{equation}
where the external time-dependent potential $V^\eps$ is given by
\begin{equation}
  \label{eq:Veps}
  V^\eps(t,y)= \frac{1}{\eps}\(V(x(t)+
\sqrt{\eps}y)-V(x(t))-\sqrt{\eps}\<\nabla V(x(t)),y\>\).
\end{equation}
This potential corresponds to the first term of a Taylor expansion of
$V$ about the point $q(t)$, and we naturally introduce 
$u=u(t,y)$ solution to 
\begin{equation}\label{eq:ulin}
i\d_tu+\frac{1}{2}\Delta u=\frac{1}{2}\< Q(t)y,y\> u\quad
;\quad u(0,y)=a(y),
\end{equation}
where
\begin{equation*}
  Q(t):= \nabla^2 V\(q(t)\), \quad \text{so that } \frac{1}{2}\<
  Q(t)y,y\>  = \lim_{\eps \to 0} V^\eps(t,y). 
\end{equation*}
The obvious candidate to approximate the initial wave function
$\psi^\eps$ is then:
\begin{equation}
  \label{eq:phi}
  \varphi^\eps(t,x)=\eps^{-d/4} u
\left(t,\frac{x-q(t)}{\sqrt\eps}\right)e^{i\left(S(t)+p(t)\cdot
    (x-q(t))\right)/\eps}.
\end{equation}
Indeed, it can be proven (see
e.g. \cite{BGP99,BR02,CoRoBook,Hag80,HaJo00,HaJo01}) that there
exists 
$C>0$ independent of $\eps$ such that
\begin{equation*}
  \|\psi^\eps(t,\cdot)-\varphi^\eps (t,\cdot)\|_{L^2(\R^d)}\le
  C\sqrt\eps e^{Ct}. 
\end{equation*}
Therefore, $\varphi^\eps$ is a good approximation of $\psi^\eps$ at least up to time
of order $c\ln\frac{1}{\eps}$ (Ehrenfest time). 

\subsubsection{Nonlinear case}
\label{sec:nonlinear}
 When adding a nonlinear term to \eqref{eq:lin}, one has to be
 cautious about the size of the solution, which rules the importance
 of the nonlinear term. To simplify the discussions, we restrict our
 analysis to the case of a gauge invariant, defocusing, power nonlinearity,
 $|\psi^\eps|^{2\si}\psi^\eps$. We choose to measure the importance of
 nonlinear effects not directly through the size of the initial data,
 but through an $\eps$-dependent coupling factor: we keep the initial
 datum \eqref{eq:ci} (with an $L^2$-norm independent of $\eps$), and
 consider
 \begin{equation*}
   i\eps\d_t \psi^\eps + \frac{\eps^2}{2}\Delta \psi^\eps =
   V(x)\psi^\eps + \eps^\alpha|\psi^\eps|^{2\si}\psi^\eps.
 \end{equation*}
Since the nonlinearity is homogeneous, this approach is equivalent to
considering $\alpha=0$, up to multiplying the initial datum by
$\eps^{\alpha/(2\si)}$. 
We assume $\si>0$, with $\si<2/(d-2)$ if $d\ge 3$: for $a\in \Sigma$,
defined by
\begin{equation*}
  \Sigma = \{f\in H^1(\R^d),\quad x\mapsto \<x\> f(x)\in
  L^2(\R^d)\},\quad \<x\>=\(1+|x|^2\)^{1/2},
\end{equation*}
we have, for fixed $\eps>0$, $\psi^\eps_{\mid t=0}\in \Sigma$, and the
Cauchy problem is globally well-posed, $\psi^\eps\in C(\R_t;\Sigma)$
(see e.g. \cite{Ca11}). It was established in 
\cite{CaFe11} that the value 
\begin{equation*}
  \alpha_c = 1+\frac{d\si}{2}
\end{equation*}
is critical in terms of the effect of the nonlinearity in the
semi-classical limit $\eps\to 0$. If $\alpha>\alpha_c$, then 
$\varphi_{\rm lin}^\eps$, given by \eqref{eq:ulin}-\eqref{eq:phi},  is
still a good approximation of $\psi^\eps$ at least up to time
of order $c\ln\frac{1}{\eps}$. On the other hand, if
$\alpha=\alpha_c$, nonlinear effects alter the behavior of $\psi^\eps$
at leading order, through its envelope only. Replacing \eqref{eq:ulin}
by 
\begin{equation}\label{eq:u}
i\d_tu+\frac{1}{2}\Delta u=\frac{1}{2}\< Q(t)y,y\> u+|u|^{2\si}u,
\end{equation}
and keeping the relation \eqref{eq:phi}, $\varphi^\eps$ is now a good
approximation of $\psi^\eps$. In \cite{CaFe11} though, the time of
validity of the approximation is not always proven to be of order at
least $c\ln\frac{1}{\eps}$, sometimes shorter time scales (of the
order $c\ln\ln\frac{1}{\eps}$) have to be considered, most likely for
technical reasons only. Some of these restrictions have been removed in
\cite{Ha13}, by considering decaying external
potentials $V$.

\subsection{Linear scattering theory and coherent states}
\label{sec:line-scatt-theory}

We now consider the aspect of large time, and instead of prescribing
$\psi^\eps$ at $t=0$ (or more generally at some finite time), we
impose its behavior at $t=-\infty$.
 In the linear case \eqref{eq:lin},
there are several results addressing the question mentioned above,
considering different forms of asymptotic states at $t=-\infty$.
Before describing them, we recall important facts concerning quantum
and classical scattering. 

\subsubsection{Quantum scattering}
\label{sec:quantum-scattering}

Throughout this paper, we assume that the external potential is
short-range, and satisfies the following properties:
\begin{hyp}\label{hyp:V}
  We suppose that $V$ is smooth and real-valued, $V\in
  C^\infty(\R^d;\R)$. In addition, it is short range in the following
  sense: there exists $\mu>1$ such that
  \begin{equation}
    \label{eq:short}
    |\d^\alpha V(x)|\le \frac{C_\alpha}{(1+|x|)^{\mu+|\alpha|}},\quad
   \forall \alpha\in \N^d. 
  \end{equation}
\end{hyp}
Our final result is established under the stronger condition
$\mu>2$ (a condition which is needed in several steps of the proof), but
some results are established under the mere assumption
$\mu>1$. Essentially, the analysis of the approximate solution is valid
for $\mu>1$
(see Section~\ref{sec:class}), while the rest of the analysis requires $\mu>2$. 
\smallbreak

Denote by 
\begin{equation*}
H_0^\eps= -\frac{\eps^2}{2}\Delta\quad \text{and}\quad
H^\eps=-\frac{\eps^2}{2}\Delta+V(x) 
\end{equation*}
the underlying Hamiltonians. For fixed $\eps>0$, the (linear) wave
operators are given by
\begin{equation*}
  W_\pm^\eps = \lim_{t\to \pm \infty}e^{i\frac{t}{\eps}H^\eps}e^{-i\frac{t}{\eps}H^\eps_0},
\end{equation*}
and the (quantum) scattering operator is defined by
\begin{equation*}
  S^\eps_{\rm lin} = \(W_+^\eps\)^* W_-^\eps. 
\end{equation*}
See for instance \cite{DG}.

\subsubsection{Classical scattering}
\label{sec:classical-scattering}

Let $V$ satisfying Assumption~\ref{hyp:V}. 
For $(q^-,p^-)\in \R^d\times \R^d$, we consider the classical
trajectories  $(q(t),p(t))$ defined by \eqref{eq:hamil}, 
along with the prescribed asymptotic behavior as $t\to -\infty$:
\begin{equation}
  \label{eq:CI-hamilton}
  \lim_{t\to -\infty}\left| q(t)-p^- t -q^-\right| = \lim_{t\to
    -\infty} |p(t)-p^-|=0. 
\end{equation}
The existence and uniqueness of such a trajectory can be found in
e.g. \cite{DG,ReedSimon3}, provided that $p^-\not =0$. Moreover, there
exists a closed set $\mathcal N_0$ of Lebesgue measure zero in
$\R^{2d}$ such that for all $(q^-,p^-)\in \R^{2d}\setminus \mathcal
N_0$, there exists $(q^+,p^+)\in \R^d\times
  \(\R^d\setminus\{0\}\)$ such that
\begin{equation*}
  \lim_{t\to +\infty}\left| q(t)-p^+ t -q^+\right| = \lim_{t\to
    +\infty} |p(t)-p^+|=0. 
\end{equation*}
The classical scattering operator is $S^{\rm cl}:(q^-,p^-)\mapsto
(q^+,p^+)$. Choosing  $(q^-,p^-)\in \R^{2d}\setminus \mathcal
N_0$ implies that the following assumption is satisfied:
\begin{assumption}\label{hyp:flot}
  The asymptotic center in phase space, $(q^-,p^-)\in \R^d\times
  \(\R^d\setminus\{0\}\)$ is such that the classical scattering
  operator is well-defined, 
  \begin{equation*}
    S^{\rm cl}(q^-,p^-)=
(q^+,p^+),\quad p^+\not =0,
  \end{equation*}
and the classical action has  limits as $t\to \pm\infty$:
\begin{equation*}
  \lim_{t\to -\infty}\left|S(t)-t\frac{|p^-|^2}{2}\right| =
  \lim_{t\to +\infty}\left|S(t)-t\frac{|p^+|^2}{2}-S_+\right|  =0,
\end{equation*}
for some $S_+\in \R$. 
\end{assumption}
 
\subsubsection{Some previous results}
\label{sec:some-prev-results}

It seems that the first mathematical result involving both the
semi-classical and large time limits appears in
\cite{GV79mean}, where the classical field limit of non-relativistic
many-boson theories is studied in space dimension $d\ge 3$. 

In
\cite{Yajima79}, the
case of a short range potential (Assumption~\ref{hyp:V}) is
considered, with asymptotic states 
under the form of semi-classically concentrated functions,
\begin{equation*}
  e^{-i\frac{\eps t}{2}\Delta}\psi^\eps(t)_{\mid t =-\infty}
  =\frac{1}{\eps^{d/2}}\widehat f\(\frac{x-q^-}{\eps}\),\quad f\in L^2(\R^d),
\end{equation*}
where $\widehat f$ denotes the standard Fourier transform (whose
definition is independent of $\eps$). The main result from
\cite{Yajima79} shows that the semi-classical limit for $S^\eps_{\rm lin}$
can be expressed in terms of the classical scattering operator, of the
classical action, and of
the Maslov index associated to each classical trajectory. We refer to
\cite{Yajima79} for a precise statement, and to \cite{Yaj81} for the
case of long range potentials, requiring modifications of the
dynamics. 
\smallbreak

In \cite{Hag81,HaJo00}, coherent  states are considered,
\begin{equation}
  \label{eq:asym-state}
 e^{-i\frac{\eps t}{2}\Delta}\psi^\eps(t)_{\mid t =-\infty}=
  \frac{1}{\eps^{d/4}}u_-\(\frac{x-q^-}{\sqrt\eps}\) 
  e^{ip^-\cdot (x-q^-)/\eps+iq^-\cdot p^-/(2\eps)}=:\psi_-^\eps(x).
\end{equation}
More precisely, in \cite{Hag81,HaJo00}, the asymptotic state $u_-$ is assumed
to be a complex Gaussian function. Introduce the notation
\begin{equation*}
  \delta(t) = S(t)-\frac{q(t)\cdot p(t) - q^-\cdot p^-}{2}.
\end{equation*}
Then Assumption~\ref{hyp:flot} implies that there exists $\delta^+\in
\R$ such that
\begin{equation*}
  \delta(t)\Tend t {-\infty}0\quad \text{and}\quad \delta(t)\Tend t
  {+\infty} \delta^+. 
\end{equation*}
 In 
\cite{CoRoBook,HaJo00}, we find the following general result  (an asymptotic
expansion in powers of $\sqrt\eps$ is actually given, but we stick to
the first term to ease the presentation):
\begin{theorem}\label{theo:version-lineaire}
  Let Assumptions~\ref{hyp:V} and \ref{hyp:flot} be satisfied, and let
  \begin{equation*}
    u_-(y) = a_- \exp \(\frac{i}{2}\<\Gamma_-y,y\>\),
  \end{equation*}
where $a_-\in \C$ and $\Gamma_-$ is a complex symmetric $d\times d$
matrix whose 
imaginary part is positive and non-degenerate. Consider $\psi^\eps$
solution to \eqref{eq:lin}, with \eqref{eq:asym-state}. Then the
following asymptotic expansion holds in $L^2(\R^d)$:
\begin{equation*}
  S^\eps_{\rm lin} \psi_-^\eps = \frac{1}{\eps^{d/4}}e^{i\delta^+/\eps} e^{ip^+\cdot
    (x-q^+)/\eps+iq^+\cdot p^+/(2\eps)} \hat R(G_+)
  u_-\(\frac{x-q^+}{\sqrt\eps}\) +\O(\sqrt\eps),
\end{equation*}
where $\hat R(G_+)$ is the metaplectic transformation associated to
$G_+ = \frac{\d (q^+,p^+)}{\d(q^-,p^-)}$. 
\end{theorem}
As a corollary, our main result yields another interpretation of the above
statement. It turns out that a complete scattering theory is available
for \eqref{eq:ulin}. As a particular case of
Theorem~\ref{theo:scatt-class} (which addresses the nonlinear case), given $u_-\in
\Sigma$, there exist a unique $u\in C(\R;\Sigma)$ solution to
\eqref{eq:ulin} and a unique $u_+\in
\Sigma$ such that  
\begin{equation*}
  \|e^{-i\frac{t}{2}\Delta}u(t)-u_\pm \|_\Sigma \Tend t {\pm \infty}
  0. 
\end{equation*}
Then in the above theorem (where $u_-$ is restricted to be a Gaussian), we have
\begin{equation*}
  u_+ = \hat R(G_+)
  u_-.
\end{equation*}
Finally, we mention in passing the paper \cite{NierENS}, where similar
issues and results are obtained for
\begin{equation*}
  i\eps\d_t \psi^\eps + \frac{\eps^2}{2}\Delta \psi^\eps =
  V\(\frac{x}{\eps}\) \psi^\eps + U(x)\psi^\eps,
\end{equation*}
for $V$ a short-range potential, and $U$ is bounded as well as its
derivatives. The special scaling in $V$ implies that initially
concentrated waves (at scaled $\eps$) first undergo the effects of $V$,
then exit a time layer of order $\eps$, through which the main action of $V$
corresponds to the above quantum scattering operator (but with $\eps=1$
due to the new scaling in the equation). Then, the action of $V$
becomes negligible, and the propagation of the wave is dictated by
the classical dynamics associated to $U$.

\subsection{Main results}
\label{sec:main}
 We now consider the nonlinear equation
\begin{equation}
  \label{eq:psi-eps}
  i\eps\d_t \psi^\eps +\frac{\eps^2}{2}\Delta \psi^\eps = V(x)\psi^\eps +
  \eps^\alpha|\psi^\eps|^{2\si}\psi^\eps,
\end{equation}
along with asymptotic data \eqref{eq:asym-state}. We first prove that
for fixed $\eps>0$, a scattering theory is available for
\eqref{eq:psi-eps}: at this stage, the value of $\alpha$ is naturally
irrelevant, as well as the form \eqref{eq:asym-state}.
To establish a large data scattering theory for \eqref{eq:psi}, we
assume that the attractive part of the potential,
\begin{equation*}
  (\d_r V(x))_+=
  \(\frac{x}{|x|}\cdot \nabla V(x)\)_+
\end{equation*}
 is not too large, where $f_+=\max (0,f)$ for any real number $f$.
\begin{theorem}\label{theo:scatt-quant}
  Let $d\ge 3$, $\frac{2}{d}<\si<\frac{2}{d-2}$, and $V$ satisfying
  Assumption~\ref{hyp:V} for some $\mu>2$. There exists $M=M(\mu,d)$ such
  that if the attractive part of the potential $(\d_r V)_+$ satisfies
  \begin{equation*}
    (\d_r V(x))_+\le \frac{M}{(1+|x|)^{\mu+1}},\quad \forall x\in
    \R^d,
  \end{equation*}
one can define a
  scattering operator for \eqref{eq:psi} in $H^1(\R^d)$: for
  all $\psi_-^\eps\in H^1(\R^d)$, there exist a unique $\psi^\eps\in
  C(\R;H^1(\R^d))$ solution to \eqref{eq:psi} and a unique $\psi_+^\eps\in
  H^1(\R^d)$ such that
  \begin{equation*}
    \|\psi^\eps(t)-e^{i\frac{\eps t}{2}\Delta}\psi_\pm^\eps\|_{H^1(\R^d)}\Tend t {\pm
      \infty} 0.
  \end{equation*}
The (quantum)  scattering operator is the map
$S^\eps:\psi_-^\eps\mapsto \psi_+^\eps$.
\end{theorem}
We emphasize the fact that several recent results address the same
issue, under various assumptions on the external potential $V$:
\cite{ZhZh14} treats the case where $V$ is an inverse square (a
framework which is ruled out in our contribution), while in
\cite{CaDa-p}, the potential is more general than merely inverse
square. In \cite{CaDa-p}, a magnetic field is also included, and the
Laplacian is perturbed with variable coefficients. We make more
comparisons with \cite{CaDa-p} in Section~\ref{sec:quant}. 
\smallbreak

The second result of this paper concerns the scattering theory for the
envelope equation:

\begin{theorem}\label{theo:scatt-class}
   Let $d\ge 1$, $\frac{2}{d}\le \si<\frac{2}{(d-2)_+}$, and $V$ satisfying
  Assumption~\ref{hyp:V} for some $\mu>1$. One can define a
  scattering operator for \eqref{eq:u} in $\Sigma$: for
  all $u_-\in \Sigma$, there exist a unique $u\in
  C(\R;\Sigma)$ solution to \eqref{eq:u} and a unique $u_+\in
  \Sigma$ such that
  \begin{equation*}
    \|e^{-i\frac{t}{2}\Delta}u(t)-u_\pm\|_{\Sigma}\Tend t {\pm
      \infty} 0.
  \end{equation*}
\end{theorem}
As mentioned above, the proof includes the construction of a linear
scattering operator, comparing the dynamics associated to
\eqref{eq:ulin} to the free dynamics $e^{i\frac{t}{2}\Delta}$. In the
above formula, we have incorporated the information that
$e^{i\frac{t}{2}\Delta}$ is unitary on $H^1(\R^d)$, but \emph{not on
}$\Sigma$ (see e.g. \cite{CazCourant}). 
\smallbreak

We can now state the nonlinear analogue to
Theorem~\ref{theo:version-lineaire}. Since
Theorem~\ref{theo:scatt-quant} requires $d\ge 3$, we naturally have to
make this assumption. On the other hand, we will need the
approximate envelope $u$ to be rather smooth, which requires a smooth
nonlinearity, $\si\in \N$. Intersecting this property with the
assumptions of Theorem~\ref{theo:scatt-quant} leaves only one case:
$d=3$ and $\si=1$, that is \eqref{eq:1}, up to the scaling. We will
see in Section~\ref{sec:cv} that considering $d=3$ is also crucial,
since the argument uses dispersive estimates which are known only in
the three-dimensional case for $V$ satisfying Assumption~\ref{hyp:V}
with $\mu>2$ (larger values for $\mu$ could be considered in higher
dimensions, though).  Introduce
the notation
\begin{equation*}
  \Sigma^k=\{ f\in H^k(\R^d),\quad x\mapsto |x|^k f(x)\in
  L^2(\R^d)\}. 
\end{equation*}

\begin{theorem}\label{theo:cv}
 Let Assumptions~\ref{hyp:V} and \ref{hyp:flot} be satisfied, with
 $\mu>2$ and $V$ as in Theorem~\ref{theo:scatt-quant}. Consider 
 $\psi^\eps$ solution to 
 \begin{equation*}
    i\eps\d_t \psi^\eps +\frac{\eps^2}{2}\Delta \psi^\eps =
  V(x)\psi^\eps + \eps^{5/2}|\psi^\eps|^2 \psi^\eps,\quad (t,x)\in \R\times \R^3,
 \end{equation*}
and such that \eqref{eq:asym-state} holds, with $u_-\in \Sigma^7$.
Then the
following asymptotic expansion holds in $L^2(\R^3)$:
\begin{equation}\label{eq:asym-finale}
  S^\eps \psi_-^\eps = \frac{1}{\eps^{3/4}}e^{i\delta^+/\eps} e^{ip^+\cdot
    (x-q^+)/\eps+iq^+\cdot p^+/(2\eps)} u_+\(\frac{x-q^+}{\sqrt\eps}\)
  +\O(\sqrt\eps), 
\end{equation}
where $S^\eps$ is given by Theorem~\ref{theo:scatt-quant} and $u_+$
stems from Theorem~\ref{theo:scatt-class}. 
\end{theorem}
\begin{remark}
  In the subcritical case, that is if we consider
 \begin{equation*}
    i\eps\d_t \psi^\eps +\frac{\eps^2}{2}\Delta \psi^\eps =
  V(x)\psi^\eps + \eps^{\alpha}|\psi^\eps|^2 \psi^\eps,\quad (t,x)\in \R\times \R^3,
 \end{equation*}
along with \eqref{eq:asym-state}, for some $\alpha>5/2$, the argument
of the proof shows that \eqref{eq:asym-finale} remains true, but with $u_+$ given
by the scattering operator associated to \eqref{eq:ulin} (as opposed
to \eqref{eq:u}), that is, the same conclusion as in
Theorem~\ref{theo:version-lineaire} when $u_-$ is a Gaussian. 
\end{remark}
As a corollary of the proof of the above result, and of the analysis
from \cite{CaFe11}, we infer:
\begin{corollary}[Asymptotic decoupling]\label{cor:decoupling}
  Let Assumption~\ref{hyp:V} be satisfied, with
 $\mu>2$ and $V$ as in Theorem~\ref{theo:scatt-quant}. Consider 
 $\psi^\eps$ solution to 
 \begin{equation*}
    i\eps\d_t \psi^\eps +\frac{\eps^2}{2}\Delta \psi^\eps =
  V(x)\psi^\eps + \eps^{5/2}|\psi^\eps|^2 \psi^\eps,\quad (t,x)\in \R\times \R^3,
 \end{equation*}
with initial datum
\begin{equation*}
  \psi^\eps(0,x) = \sum_{j=1}^N\frac{1}{\eps^{3/4}}a_j\(\frac{x-q_{0j}}{\sqrt\eps}\)
  e^{ip_{0j}\cdot (x-q_{0j})/\eps}=:\psi_0^\eps(x),
\end{equation*}
where $N\ge 2$, $q_{0j},p_{0j}\in \R^3$, $p_{0j}\not =0$ so that scattering   is
available as
$t\to +\infty$ for  $(q_j(t),p_j(t))$, in the sense of
Assumption~\ref{hyp:flot}, and $a_j\in \Sch(\R^3)$. We suppose
$(q_{0j},p_{0j})\not =(q_{0k},p_{0k})$ for $j\not =k$. Then we have
the uniform estimate: 
\begin{equation*}
  \sup_{t\in \R}\left\| \psi^\eps(t) - \sum_{j=1}^N
    \varphi_j^\eps(t)\right\|_{L^2(\R^3)} \Tend \eps
  0 0 ,
\end{equation*}
where $\varphi_j^\eps$ is the approximate solution with the $j$-th
wave packet as an initial datum. As a consequence, the 
 asymptotic expansion holds in $L^2(\R^3)$, as $\eps \to 0$:
\begin{equation*}
  \(W^\eps_\pm\)^{-1} \psi_0^\eps =\sum_{j=1}^N
  \frac{1}{\eps^{3/4}}e^{i\delta_{j}^\pm/\eps} e^{ip_{j}^\pm\cdot 
    (x-q_{j}^\pm)/\eps+iq_{j}^\pm\cdot p_{j}^\pm/(2\eps)}
  u_{j\pm}\(\frac{x-q_{j}^\pm}{\sqrt\eps}\) 
  +o(1), 
\end{equation*}
where the inverse wave operators $\(W^\eps_\pm\)^{-1} $ stem from
Theorem~\ref{theo:scatt-quant}, the $u_{j\pm}$'s 
are the asymptotic states emanating from $a_j$, and 
\begin{equation*}
  \delta_{j}^\pm = \lim_{t\to \pm\infty}\(S_j(t) - \frac{q_j(t)\cdot
    p_j(t)-q_{0j}\cdot p_{0j}}{2}\)\in \R. 
\end{equation*}
\end{corollary}
\begin{remark}
  In the case $V=0$, the approximation by wave packets is actually
  exact, since then $Q(t)\equiv 0$, hence $u^\eps=u$. For one wave
  packet, Theorem~\ref{theo:cv} 
  becomes empty, since it is merely a rescaling. On the other hand,
  for two initial wave packets, even in the case $V=0$,
  Corollary~\ref{cor:decoupling} brings some information, reminiscent
  of profile decomposition. More precisely, define $u^\eps$ by
  \eqref{eq:chginc}, and choose (arbitrarily) to privilege the trajectory
  $(q_1,p_1)$. The Cauchy problem is then equivalent to 
  \begin{equation*}
\left\{
\begin{aligned}
    &i\d_t u^\eps+\frac{1}{2}\Delta u^\eps = |u^\eps|^2 u^\eps,\\
&    u^\eps(0,y) = a_1(y) + a_2\( y +\frac{q_{01}-q_{02}}{\sqrt\eps}\)
    e^{ip_{02}\cdot \delta q_0/\eps -i\delta p_0\cdot y/\sqrt\eps},
  \end{aligned}
\right.
\end{equation*}
where we have set $\delta p_0 = p_{01}-p_{02}$ and $\delta q_0
=q_{01}-q_{02}$. 
Note however that the initial datum is uniformly bounded in
$L^2(\R^3)$, but in no $H^s(\R^3)$ for $s>0$ (if $p_{01}\not =
p_{02}$), while the equation is 
$\dot H^{1/2}$-critical, Therefore, even in the case
$V=0$, 
Corollary~\ref{cor:decoupling} does not seem to be a consequence of
profile decompositions like in
e.g. \cite{DuHoRo08,Keraani01,MerleVega98}. In view of
\eqref{eq:hamil}, the approximation provided by
Corollary~\ref{cor:decoupling} reads, in that case:
\begin{equation*}
  u^\eps(t,y) = u_1(t,y) + u_2\(t,y
  +\frac{t \delta p_0+\delta q_0}{\sqrt\eps}\)
  e^{i\phi_2^\eps(t,y)}+o(1)\quad \text{in }L^\infty(\R;L^2(\R^3)),
\end{equation*}
where the phase shift is given by
\begin{align*}
  \phi^\eps_2(t,y) &= \frac{1}{\eps}p_{02}\cdot \( t\delta p_0+\delta
  q_0\) -\frac{1}{\sqrt\eps}\delta p_0\cdot y +\frac{t}{2\eps}
  \( |p_{02}|^2-|p_{01}|^2\) \\
&= \frac{1}{\eps}p_{02}\cdot \delta
  q_0 -\frac{1}{\sqrt\eps}\delta p_0\cdot y -\frac{t}{2\eps}|\delta
  p_0|^2. 
\end{align*}
\end{remark}

\noindent {\bf Notation.} We write $a^\eps(t)\lesssim b^\eps(t)$
whenever there exists $C$ independent of $\eps\in (0,1]$ and $t$ such
that $a^\eps(t)\le C b^\eps(t)$.

\section{Spectral properties and consequences}
\label{sec:spectral}

In this section, we derive some useful properties for the Hamiltonian
\begin{equation*}
  H=-\frac{1}{2}\Delta +V.
\end{equation*}
Since the dependence upon $\eps$ is not addressed in this
section, we assume $\eps=1$.
\smallbreak

First, it follows for instance from \cite{Mourre} that
Assumption~\ref{hyp:V} implies that $H$ has no
singular spectrum. Based on Morawetz estimates, we show that $H$ has
no eigenvalue, provided that the  attractive part of $V$ is
sufficiently small. Therefore, the spectrum of $H$ is
purely absolutely continuous. 
Finally, again if  the  attractive part of $V$ is
sufficiently small, zero is not a resonance of $H$, so Strichartz
estimates are available for $e^{-itH}$.

\subsection{Morawetz estimates and a first consequence}
\label{sec:morawetz}

In this section, we want to treat both linear and nonlinear equations,
so we consider
\begin{equation}
  \label{eq:psi-gen}
  i\d_t \psi +\frac{1}{2}\Delta \psi = V\psi + \lambda
  |\psi|^{2\si}\psi,\quad \l \in \R.
\end{equation}
Morawetz estimate in the linear case $\l=0$ will show the absence of
eigenvalues. In the nonlinear case $\l>0$, these estimates will be a
crucial tool for prove scattering in the quantum case. 
 The following lemma and its proof are essentially a rewriting of the
 presentation from \cite{BaRuVe06}.  
\begin{proposition}[Morawetz inequality]\label{prop:Morawetz}
   Let $d\ge 3$, and $V$ satisfying
  Assumption~\ref{hyp:V} for some $\mu>2$. There exists $M=M(\mu,d)>0$ such
  that if the attractive part of the potential satisfies
  \begin{equation*}
    (\d_r V(x))_+ \le \frac{M}{(1+|x|)^{\mu+1}},\quad \forall x\in
    \R^d,
  \end{equation*}
then any solution $\psi\in L^\infty(\R;H^1(\R^d))$ to \eqref{eq:psi-gen}
satisfies
\begin{equation}\label{eq:morawetz}
 \l  \iint_{\R\times \R^d}\frac{|\psi(t,x)|^{2\si+2}}{|x|}dtdx +
 \iint_{\R\times \R^d}\frac{|\psi(t,x)|^{2}}{(1+|x|)^{\mu+1}}dtdx\lesssim 
  \|\psi\|_{L^\infty(\R;H^1)}^2. 
\end{equation}
\end{proposition}
In other words, the main obstruction to global dispersion for $V$
comes from $(\d_r V)_+$, which is the attractive contribution of $V$
in classical trajectories, while $(\d_r V)_-$ is the repulsive part,
which does not ruin the dispersion associated to $-\Delta$ (it may
 reinforce it, see e.g. \cite{CaDCDS}, but repulsive
potentials do not necessarily improve the dispersion, see \cite{GoVeVi06}). 
\begin{proof}
  The proof follows standard arguments, based on virial identities
  with a suitable weight. We resume the main steps of the
  computations, and give more details on the choice of the weight in
  our context. For a real-valued function $h(x)$, we compute, for $\psi$ solution
  to \eqref{eq:psi},
  \begin{equation*}
    \frac{d}{dt}\int h(x)|\psi(t,x)|^2dx = \IM \int \bar \psi(t,x) \nabla
    h(x)\cdot \nabla \psi(t,x)dx,
  \end{equation*}
  \begin{equation}
    \label{eq:viriel}
    \begin{aligned}
        \frac{d}{dt}\IM \int \bar \psi(t,x) \nabla
    h(x)\cdot \nabla \psi(t,x)dx &= \int \nabla \bar \psi(t,x)\cdot
    \nabla^2h(x)\nabla \psi(t,x)dx \\
-\frac{1}{4}\int |\psi(t,x)|^2 &\Delta^2
    h(x)dx -\int |\psi(t,x)|^2\nabla V\cdot \nabla h(x)dx\\
 &   +\frac{\l\si}{\si+1}\int |\psi(t,x)|^{2\si+2}\Delta h(x) dx. 
 \end{aligned}
   \end{equation}
In the case $V=0$, the standard choice is $h(x)=|x|$, for which
\begin{equation*}
  \nabla h=\frac{x}{|x|},\quad \nabla^2_{jk}h =
  \frac{1}{|x|}\(\delta_{jk}-\frac{x_jx_k}{|x|^2}\),\quad \Delta h\ge
  \frac{d-1}{h},\quad \text{and }\Delta^2 h\le 0\text{ for }d\ge 3. 
\end{equation*}
This readily yields Proposition~\ref{prop:Morawetz} in the repulsive case
$\d_r V\le 0$, since $\nabla h\in L^\infty$. 
\smallbreak

In the same spirit as in \cite{BaRuVe06}, we proceed by perturbation to
construct a suitable weight when the attractive part of the potential
is not too large. We seek a priori a radial weight, $h=h(|x|)\ge 0$, so we
have 
\begin{align*}
 & \Delta h = h'' +\frac{d-1}{r} h',\\
& \Delta^2 h =  h^{(4)}
  +2\frac{d-1}{r} h^{(3)} +\frac{(d-1)(d-3)}{r^2}h'' -
  \frac{(d-1)(d-3)}{r^3}h',\\
&\nabla^2_{jk} h = \frac{1}{r}\(\delta_{jk}-\frac{x_jx_k}{r^2}\) h'
+\frac{x_jx_k}{r^2}h''. 
\end{align*}
We construct a function $h$ such that $h',h''\ge 0$, so the condition
$\nabla^2 h\ge 0$ will remain. The goal is then to construct a radial
function $h$ such that the second line in \eqref{eq:viriel} is
non-negative, along with $\Delta h \ge \eta/|x|$ for some $\eta>0$.
\smallbreak

\noindent {\bf Case $d=3$.} In this case, the expression for $\Delta^2
h$ is simpler, and the above conditions read
\begin{align*}
 &\frac{1}{4} h^{(4)}
  +\frac{1}{r} h^{(3)} + \nabla V(x)\cdot \nabla h\le 0,\\
& h''+\frac{2}{r}h'\ge \frac{\eta}{r},\quad h',h''\ge 0. 
\end{align*}
Since we do not suppose a priori that $V$ is a radial potential, the
first condition is not rigorous. We actually use the fact that for
$h'\ge 0$, Assumption~\ref{hyp:V} implies
\begin{equation*}
  \nabla V(x)\cdot \nabla h \le \(\d_r V(x)\)_+ h'(r)\le
  \frac{M}{(1+r)^{\mu+1}}h'(r). 
\end{equation*}
To achieve our goal, it is therefore sufficient to require:
\begin{align}
 \label{eq:h1}&\frac{1}{4} h^{(4)}
  +\frac{1}{r} h^{(3)} +  \frac{M}{(1+r)^{\mu+1}}h' \le 0,\\
\label{eq:h2}& h''+\frac{2}{r}h'\ge \frac{\eta}{r},\quad h'\in L^\infty(\R_+), \
h',h''\ge 0. 
\end{align}
In view of \eqref{eq:h2}, we seek
\begin{equation*}
  h'(r) = \eta +\int_0^r h''(\rho)d\rho.
\end{equation*}
Therefore, if $h''\ge 0$ with $h''\in L^1(\R_+)$, \eqref{eq:h2} will
be automatically fulfilled. We now turn to \eqref{eq:h1}. Since we
want $h'\in L^\infty$, we may even replace $h'$ by a constant in
\eqref{eq:h1}, and solve, for $C>0$, the ODE
\begin{equation*}
  \frac{1}{4} h^{(4)}
  +\frac{1}{r} h^{(3)} +  \frac{C}{(1+r)^{\mu+1}}=0.
\end{equation*}
We readily have
\begin{equation*}
  h^{(3)}(r) = -\frac{4C}{r^4}\int_0^r\frac{\rho^4}{(1+\rho)^{\mu+1}}d\rho,
\end{equation*}
along with the properties $h^{(3)}(0)=0$, 
\begin{equation*}
  h^{(3)}(r)\Eq r \infty -\frac{k}{r^{\min (\mu,4)}},\quad \text{for some }k>0.
\end{equation*}
It is now natural to set
\begin{equation*}
  h''(r) = -\int_r^\infty h^{(3)}(\rho)d\rho,
\end{equation*}
so we have $h''\in C([0,\infty);\R_+)$ and
\begin{equation*}
  h''(r) \Eq r \infty \frac{\kappa}{r^{\min (\mu-1,3)}},\quad \text{for some }\kappa>0.
\end{equation*}
This function is indeed in $L^1$ if and only if $\mu>2$. We 
define $h$ by  $h(r)= \int_0^r h'(\rho)d\rho$,
\begin{equation}\label{eq:h3}
   h^{(3)}(r) = -\frac{K}{r^4}\int_0^r\frac{\rho^4}{(1+\rho)^{\mu+1}}d\rho,
\end{equation}
for some $K>0$, $h''$ and $h'$ being given by the above relations:
\eqref{eq:h2} is satisfied for any value of $K>0$, and \eqref{eq:h1}
boils down to an inequality of the form
\begin{equation}\label{eq:M}
  -\frac{K}{4} +M\(\eta +C(\mu)K\)\le 0,
\end{equation}
where $C(\mu)$ is proportional to 
\begin{equation*}
 \frac{1}{K} \|h'\|_{L^\infty} = \int_0^\infty \int_r^\infty
 \frac{1}{\rho^4}\int_0^\rho \frac{s^4}{(1+s)^{\mu+1}}dsd\rho dr.
\end{equation*}
We infer that \eqref{eq:h3} is satisfied for $K\gg \eta$, provided
that $M<\frac{1}{4C(\mu)}$. Note then that by construction, we may also
require
\begin{equation*}
  \frac{1}{4}\Delta^2 h +\nabla V\cdot \nabla h\le \frac{-c_0}{(1+|x|)^{\mu+1}},
\end{equation*}
for $c_0>0$ morally very small. 
\smallbreak

\noindent {\bf Case $d\ge 4$.} Resume the above reductions, pretending
that the last two terms in $\Delta^2 h$ are not present: \eqref{eq:h3}
just becomes
\begin{equation*}
   h^{(3)}(r) = -\frac{K}{r^{2d-2}}\int_0^r\frac{\rho^{2d-2}}{(1+\rho)^{\mu+1}}d\rho,
\end{equation*}
and we see that with $h''$ and $h'$ defined like before, we have
\begin{equation*}
  rh''-h'= -\eta- \int_0^r h'' +rh''.
\end{equation*}
Since this term is negative at $r=0$ and has a non-positive
derivative, we have $rh''-h'\le 0$, so finally $\Delta^2 h\le 0$. 
\end{proof}

We infer that $H$ has no eigenvalue. Indeed, if there were an $L^2$ solution
$\psi=\psi(x)$ 
to $H\psi =E\psi$, $E\in \R$, then $\psi\in
H^2(\R^d)$, and $\psi(x)e^{-iEt}$ would be an $H^1$
solution to \eqref{eq:psi-gen} for $\l=0$. This is contradiction with
the global integrability in time from \eqref{eq:morawetz}, so
$\si_{\rm pp}(H)=\emptyset$. 

\subsection{Strichartz estimates}
\label{sec:strichartz}

In
\cite[Proposition~3.1]{BaRuVe06}, it is proved that zero is
not a resonance of $H$, but with a definition of resonance which is
not quite the definition in \cite{RodnianskiSchlag}, which contains a
result that we want to use. So we shall resume the argument.

By definition (as in \cite{RodnianskiSchlag}), zero is a resonance of
$H$, if there is a distributional solution 
$\psi\not\in L^2$, such that  $\<x\>^{-s}\psi\in L^2(\R^d)$ for all
$s>\frac{1}{2}$, to
$H\psi=0$.  
\begin{corollary}
  Under the assumptions of Proposition~\ref{prop:Morawetz}, zero is not a
  resonance of $H$.
\end{corollary}
\begin{proof}
  Suppose that zero is a resonance of $H$. Then by definition, we
  obtain a stationary  distributional solution of \eqref{eq:psi-gen} (case
  $\l=0$), $\psi= \psi(x)$, and we may assume that it is
  real-valued. Since $\Delta \psi = 
  2V\psi$, Assumption~\ref{hyp:V} implies
  \begin{equation*}
    \<x\>^{\mu-s}\Delta \psi\in L^2(\R^d),\quad \forall s>\frac{1}{2}. 
  \end{equation*}
This implies that $\nabla \psi\in L^2$, by taking for instance $s=1$ in
\begin{equation*}
  \int|\nabla \psi|^2 = -\int \<x\>^{-s}\psi \<x\>^s\Delta \psi.
\end{equation*}
By definition, for all test function $\varphi$,
\begin{equation}\label{eq:variat}
  \frac{1}{2}\int_{\R^d}\nabla \varphi(x) \cdot \nabla \psi(x)dx
  +\int_{\R^d}V(x)\varphi(x)\psi(x)dx =0.
\end{equation}
Let $h$ be the weight constructed in the proof of
Proposition~\ref{prop:Morawetz}, and consider
\begin{equation*}
  \varphi = \psi\Delta h +2\nabla \psi\cdot \nabla h. 
\end{equation*}
Since $\nabla h\in L^\infty$, $\nabla^2 h(x)=\O(\<x\>^{-1})$, and
$\nabla^3 h(x)= \O(\<x\>^{-2})$, we see
that $\varphi \in H^1$, and that this choice is allowed in
\eqref{eq:variat}. Integration by parts then yields \eqref{eq:viriel}
(where the left hand side is now zero):
\begin{equation*}
  0=\int \nabla \psi\cdot \nabla^2h \nabla \psi -\frac{1}{4}\int
  \psi^2\Delta^2 h -\int \psi^2 \nabla V\cdot \nabla h.
\end{equation*}
By construction of $h$, this implies
\begin{equation*}
  \int_{\R^d}\frac{\psi(x)^2}{(1+|x|)^{\mu+1}}dx\le 0,
\end{equation*}
hence $\psi\equiv 0$. 
\end{proof}
Therefore,
\cite[Theorem~1.4]{RodnianskiSchlag} implies non-endpoint global in
time Strichartz estimates.  In the case $d=3$, we know from
\cite{Go06} that (in view of the above spectral properties)
\begin{equation*}
  \|e^{-itH}\|_{L^1\to L^\infty}\le C |t|^{-d/2},\quad \forall t\not
  =0,
\end{equation*}
a property which is stronger than Strichartz estimates, and yields the
endpoint Strichartz estimate missing in \cite{RodnianskiSchlag}, from
\cite{KT}. On the other 
hand, this dispersive estimate does not seem to be known under
Assumption~\ref{hyp:V} with $\mu>2$ when $d\ge 4$: stronger assumptions
are always present so far (see e.g. \cite{CaCuVo09,ErGr10}). However,
endpoint Strichartz estimates for $d\ge 4$ are a consequence of
\cite[Theorem~1.1]{AFVV10}, under the assumptions of
Proposition~\ref{prop:Morawetz}. 

\begin{proposition}\label{prop:StrichartzRS}
Let $d\ge 3$. Under the assumptions of
Proposition~\ref{prop:Morawetz}, for all $(q,r)$ such that 
  \begin{equation}\label{eq:adm}
    \frac{2}{q}=d\(\frac{1}{2}-\frac{1}{r}\),\quad 2<q\le \infty,
  \end{equation}
there exists $C=C(q,d)$ such that 
\begin{equation*}
  \|e^{-itH} f\|_{L^q(\R;L^r(\R^d))}\le C \|f\|_{L^2(\R^d)},\quad
  \forall f\in L^2(\R^d). 
\end{equation*}
\end{proposition}
It is classical that this homogeneous Strichartz estimate, a duality
argument and Christ-Kiselev's Theorem imply the inhomogeneous
counterpart. For two admissible pairs $(q_1,r_1)$ and $(q_2,r_2)$
(that is, satisfying \eqref{eq:adm}), there exists $C_{q_1,q_2}$
independent of the time interval $I$ such
that if we denote by
\begin{equation*}
  R(F)(t,x) = \int_{I\cap \{s\le t\}} e^{-i(t-s)H}F(s,x)ds,
\end{equation*}
we have
\begin{equation*}
  \|R(F)\|_{L^{q_1}(I;L^{r_1}(\R^d))}\le
  C_{q_1,q_2}\|F\|_{L^{q_2'}(I;L^{r_2'}(\R^d))},\quad \forall F\in
  L^{q_2'}(I;L^{r_2'}(\R^d)). 
\end{equation*}
\smallbreak

Note that the assumption $\mu>2$ seems essentially sharp in order to have
global in time Strichartz estimates. The result remains true for $\mu
=2$ (\cite{BPST03,BPST04}), but in \cite{GoVeVi06}, the authors
prove that for repulsive potentials which are homogeneous of degree
smaller than $2$, global Strichartz estimates fail to exist.

\section{Quantum scattering}
\label{sec:quant}

In this section, we prove Theorem~\ref{theo:scatt-quant}. Since the
dependence upon $\eps$ is not measured in
Theorem~\ref{theo:scatt-quant}, we shall 
consider the case $\eps=1$, corresponding to 
\begin{equation}
  \label{eq:psi}
  i\d_t \psi +\frac{1}{2}\Delta \psi = V\psi + |\psi|^{2\si}\psi.
\end{equation}
We split the proof of Theorem~\ref{theo:scatt-quant} into two
steps. First, we solve the Cauchy problem with data prescribed at
$t=-\infty$, that is, we show the existence of wave operators. Then,
given an initial datum at $t=0$, we show that the (global) solution to
\eqref{eq:psi} behaves asymptotically like a free solution, which
corresponds to asymptotic completeness. 
\smallbreak

For each of these two steps, we first show that the nonlinearity is
negligible for large time, and then recall that the potential is
negligible for large time (linear scattering). This means that for any $\tilde \psi_-\in
H^1(\R^d)$, there exists  a unique $\psi\in 
  C(\R;H^1(\R^d))$ solution to \eqref{eq:psi}  such that
  \begin{equation*}
    \|\psi(t)-e^{-itH}\tilde \psi_-\|_{H^1(\R^d)}\Tend t {-
      \infty} 0,
  \end{equation*}
and for any $\varphi\in H^1(\R^d)$, there exist a unique  $\psi\in
  C(\R;H^1(\R^d))$ solution to \eqref{eq:psi} and a unique $\tilde\psi_+\in
  H^1(\R^d)$ such that
\begin{equation*}
    \|\psi(t)-e^{-itH}\tilde \psi_+\|_{H^1(\R^d)}\Tend t {+
      \infty} 0.
  \end{equation*}
Then, we recall that the potential $V$ is negligible for large
time. We will adopt the following notations for the propagators,
\begin{equation*}
  U(t)=e^{i\frac{t}{2}\Delta},\quad U_V(t)= e^{-itH}. 
\end{equation*}

In order to construct wave operators which show that the nonlinearity
can be neglected for large time, we shall work with an $H^1$
regularity, on the Duhamel's formula associated to \eqref{eq:psi} in
terms of $U_V$, with a prescribed asymptotic behavior as $t\to
-\infty$:
\begin{equation}
  \label{eq:duhamel-}
  \psi(t) = U_V(t)\tilde \psi_- -i\int_{-\infty}^t
  U_V(t-s)\(|\psi|^{2\si}\psi(s)\)ds. 
\end{equation}
Applying the gradient to this formulation brings up the problem of
non-commutativity with $U_V$. The worst term is actually the linear
one, $U_V(t)\tilde \psi_-$, since
\begin{equation*}
  \nabla \(U_V(t)\tilde \psi_-\) = U_V(t)\nabla \tilde \psi_-
  -i\int_0^t U_V(t-s)\((U_V(s)\tilde \psi_-)\nabla V\)ds.
\end{equation*}
Since the construction of wave operators relies on the use of
Strichartz estimates, it would be necessary to have an estimate of
\begin{equation*}
 \left\|\nabla \(U_V(t)\tilde \psi_-\)\right\|_{L^qL^r}
\end{equation*}
in terms of $\psi_-$, for admissible pairs
$(q,r)$. Proposition~\ref{prop:StrichartzRS} yields
\begin{equation*}
  \left\|\nabla \(U_V(t)\tilde \psi_-\)\right\|_{L^qL^r} \lesssim \|\nabla \tilde
  \psi_-\|_{L^2} + \|(U_V(t)\tilde \psi_-)\nabla V\|_{L^{\tilde
      q'}L^{\tilde r'}},
\end{equation*}
for any admissible pair $(\tilde q,\tilde r)$. In the last factor,
time is present only in the term $U_V(t)\tilde \psi_-$, so to be able
to use Strichartz estimates again, we need to consider $\tilde
q=2$, in which case $\tilde r=2^*:=\frac{2d}{d-2}$:
\begin{equation*}
  \|(U_V(t)\tilde \psi_-)\nabla V\|_{L^2L^{{2^*}'}}\le \|U_V(t)\tilde
  \psi_-\|_{L^2L^{2^*}}\|\nabla V\|_{L^{d/2}},
\end{equation*}
where Assumption~\ref{hyp:V} implies $\nabla V\in L^{d/2}(\R^d)$ as
soon as $\mu>1$. Using the endpoint Strichartz estimate from
Proposition~\ref{prop:StrichartzRS}, we have
\begin{equation*}
  \|U_V(t)\tilde
  \psi_-\|_{L^2L^{2^*}} \lesssim \|\tilde \psi_-\|_{L^2}, 
\end{equation*}
and we have:
\begin{lemma}\label{lem:stri2}
  Let $d\ge 3$. Under the assumptions of
  Proposition~\ref{prop:Morawetz}, for all admissible pair $(q,r)$, 
  \begin{equation*}
    \|e^{-itH}f\|_{L^q(\R;W^{1,r}(\R^d))}\lesssim \|f\|_{H^1(\R^d)}. 
  \end{equation*}
\end{lemma}
We shall rather use a vector-field, for we believe this approach may be
interesting in other contexts.

\subsection{Vector-field}
\label{sec:vector-field}

 We
introduce a vector-field which naturally commutes with $U_V$, and
is comparable with the gradient. 
\smallbreak

From Assumption~\ref{hyp:V}, $V$ is bounded, so there exists $c_0\ge
0$ such that $V+c_0\ge 0$. We shall consider the operator
\begin{equation*}
  A = \sqrt{H+c_0}=\sqrt{-\frac{1}{2}\Delta +V+c_0}.
\end{equation*}
\begin{lemma}\label{lem:A}
 Let $d\ge 3$, and $V$ satisfying Assumption~\ref{hyp:V} with
 $V+c_0\ge 0$. For every $1<r<\infty$, there exists $C_r,K_r$  such
 that for all $f\in    W^{1,r}(\R^d)$,
  \begin{equation}
    \label{eq:A}
    \|A f\|_{L^r}\le C_r
    \(\|f\|_{L^r}+\|\nabla f\|_{L^r}\)\le K_r \(\|f\|_{L^r}+\|A f\|_{L^r}\) .
  \end{equation}
\end{lemma}
\begin{proof}
  The first inequality is very close to \cite[Theorem~1.2]{AFVV10}, and the
  proof can readily be adapted. On the other hand, the second
  inequality would require the restriction $4/3<r<4$ if we followed 
  the same approach, based on  Stein's interpolation theorem (a
  similar approach for followed in e.g. \cite{KiViZh09}). We
  actually take  advantage of the smoothness of the potential $V$ to
  rather apply Calder\'on--Zygmund result on the action of
  pseudo-differential operators. 
\smallbreak

We readily check that the two functions
\begin{equation*}
  a(x,\xi)=
  \sqrt{\frac{\frac{|\xi|^2}{2}+V(x)+c_0}{1+|\xi|^2}},\quad
  b(x,\xi) = \sqrt{\frac{|\xi|^2}{\frac{|\xi|^2}{2}+V(x)+c_0+1}} ,
\end{equation*}
are symbols of order zero, in the sense
that they satisfy
\begin{equation*}
  |\d_x^\alpha\d_\xi^\beta a(x,\xi)| +  |\d_x^\alpha\d_\xi^\beta
  b(x,\xi)| \le C_{\alpha,\beta}\<\xi\>^{-|\beta|},
\end{equation*}
for all $\alpha,\beta\in \N^d$. This implies that the
pseudo-differential operators of symbol 
$a$ and $b$, respectively, are bounded on $L^r(\R^d)$, for
all $1<r<\infty$; see e.g. \cite[Theorem~5.2]{Taylor3}. In the case of
$a$, this yields the first inequality in \eqref{eq:A}, and in the case of $b$, this
yields the second inequality. 
\end{proof}

\subsection{Wave operators}
\label{sec:wave-op}

With the tools presented in the previous section, we can prove the
following result by adapting the standard proof of the case $V=0$, as
established in \cite{GV85}.
\begin{proposition}\label{prop:waveop-quant}
  Let $d\ge 3$, $\frac{2}{d}\le \si<\frac{2}{d-2}$, and $V$ satisfying
  Assumption~\ref{hyp:V} for some $\mu>2$. For
  all $\tilde \psi_-\in H^1(\R^d)$, there exists a unique 
$$\psi\in
  C((-\infty,0];H^1(\R^d))\cap
  L^{\frac{4\si+4}{d\si}}((-\infty,0);L^{2\si+2}(\R^d))$$
 solution to    \eqref{eq:psi}  such that 
  \begin{equation*} 
    \|\psi(t)-e^{-it H}\tilde\psi_-\|_{H^1(\R^d)}\Tend t {-
      \infty} 0.
  \end{equation*} 
\end{proposition}
\begin{proof}
  The main part of the proof is to prove that \eqref{eq:duhamel-} has
  a fixed point. Let
  \begin{equation*}
   q=\frac{4\si+4}{d\si}.
  \end{equation*}
The pair $(q,2\si+2)$ is admissible, in the sense that it satisfies
\eqref{eq:adm}.  
With the notation $L^\beta_TY=L^\beta(]-\infty,-T];Y)$, we
  introduce:  
  \begin{align*}
    X_T:=\Big\{ \psi\in C(]-\infty,-T];H^1)\ ;\ &\left\|
    \psi\right\|_{L^q_TL^{2\si+2}} \le  K\|\tilde
                                                  \psi_-\|_{L^2},\\
\left\|
    \nabla \psi\right\|_{L^q_TL^{,2\si+2}} \le  K\|\tilde
                                                  \psi_-\|_{H^1},\quad &
 \left\|  \psi\right\|_{L^\infty_TL^2} \le 2 \|\tilde \psi_-\|_{L^2}\,
    ,\\
 \left\| \nabla  \psi\right\|_{L^\infty_TL^2} \le K \|\tilde \psi_-\|_{H^1}
,\quad 
&\left\|  \psi\right\|_{L^q_T L^{2\si+2}} \le 2 \left\|
    U_V(\cdot)\tilde \psi_-\right\|_{L^q_T L^{2\si+2}}\Big\},
  \end{align*}
where $K$ will be chosen sufficiently large in terms of the
constants present in Strichartz estimates presented in
Proposition~\ref{prop:StrichartzRS}. Set
$r=s=2\si +2$: we have
\begin{equation*}
  \frac{1}{r'}= \frac{1}{r}+\frac{2\si}{s},\quad
\frac{1}{q'}= \frac{1}{q}+\frac{2\si}{k},
\end{equation*}
where $q\le k<\infty$ since $2/d\le
\si<2/(d-2)$. Denote by $\Phi(\psi)$ the right hand side of
\eqref{eq:duhamel-}. For $\psi\in X_T$, Strichartz estimates and H\"older
inequality yield, for all admissible pairs $(q_1,r_1)$:
\begin{align*}
  \left\| \Phi(\psi)\right\|_{L^{q_1}_T L^{r_1}} &\le C_{q_1}\|\tilde    
\psi_-\|_{L^2} +  C\left\| |\psi|^{2\si}\psi\right\|_{L^{q'}_TL^{r'}}  
  \\
&\le C_{q_1}\|\tilde \psi_-\|_{L^2}  +
C\|\psi\|_{L^k_TL^s}^{2\si}\|\psi\|_{L^q_T L^r}\\
&\le C_{q_1}\|\tilde \psi_-\|_{L^2} + C\|\psi\|_{L^q_TL^r}^{2\si\theta
  }\|\psi\|_{L^\infty_TL^r}^{2\si(1-\theta) } \|\psi\|_{L^q_T L^r} ,
\end{align*}
for some $0<\theta\le 1$, where we have used the property
$r=s=2\si+2$. Sobolev embedding and the definition of $X_T$ then imply:
\begin{align*}
 \left\| \Phi(\psi)\right\|_{L^{q_1}_T L^{r_1}} \le C_{q_1}\|\tilde
  \psi_-\|_{L^2} + C\left\|U_V(\cdot) \tilde 
  \psi_-\right\|_{L^q_TL^r}^{2\si\theta 
  }\|\psi\|_{L^\infty_TH^1}^{2\si(1-\theta) } \|\psi\|_{L^q_T L^r} .
\end{align*}
We now apply the operator $A$. Since $A$ commutes with $H$, we have
\begin{equation*}
  \left\| A\Phi(\psi)\right\|_{L^{q_1}_T L^{r_1}} \lesssim
  \|A\tilde \psi_-\|_{L^2} + \left\|
    A\(|\psi|^{2\si}\psi\)\right\|_{L^{q'}_TL^{r'}}. 
\end{equation*}
In view of Lemma~\ref{lem:A}, we have successively,
\begin{align*}
  \|A\tilde \psi_-\|_{L^2} &\lesssim \|\tilde \psi_-\|_{H^1},\\
 \left\|
    A\(|\psi|^{2\si}\psi\)\right\|_{L^{q'}_TL^{r'}}&\lesssim   \left\|
    |\psi|^{2\si}\psi\right\|_{L^{q'}_TL^{r'}} + \left\|
  \nabla   \(|\psi|^{2\si}\psi\)\right\|_{L^{q'}_TL^{r'}} \\
&\lesssim \|\psi\|_{L^k_TL^s}^{2\si}\(\|\psi\|_{L^q_T L^r} +
  \|\nabla\psi\|_{L^q_T L^r} \)\\
&\lesssim \|\psi\|_{L^k_TL^s}^{2\si}\(\|\psi\|_{L^q_T L^r} +
  \|A\psi\|_{L^q_T L^r} \).
\end{align*}
We infer along the same lines as above,
\begin{align*}
 \left\| \nabla\Phi(\psi)\right\|_{L^{q_1}_T L^{r_1}} &\lesssim
  \|\tilde \psi_-\|_{H^1} +\left\|U_V(\cdot) \tilde
  \psi_-\right\|_{L^q_TL^r}^{2\si\theta 
  }\|\psi\|_{L^\infty_TH^1}^{2\si(1-\theta) } \( 
\| \psi\|_{L^q_T L^r} + \|A\psi\|_{L^q_T L^r}\) .
 \end{align*}
We have also
\begin{align*}
\left\| \Phi(\psi)\right\|_{L^{q}_T L^{r}}& \le
  \left\|U_V(\cdot)\tilde \psi_-\right\|_{L^{q}_TL^{r}} 
+ C\left\|U_V(\cdot) \tilde
  \psi_-\right\|_{L^q_TL^r}^{2\si\theta 
  }\|\psi\|_{L^\infty_TH^1}^{2\si(1-\theta) } \|\psi\|_{L^q_T L^r}.
\end{align*}
From Strichartz estimates, $U_V(\cdot)\tilde \psi_- \in L^{q}(\R;L^{r})$, so 
\begin{equation*}
  \left\|U_V(\cdot)\tilde \psi_-\right\|_{L^q_TL^r}  \to 0\quad \text{as }T\to +\infty.
\end{equation*}
Since $\theta>0$, we infer that $\Phi$ sends $X_T$ to itself, for
$T$ sufficiently large. 
\smallbreak

We have also, for $\psi_2,\psi_1\in X_T$:
\begin{align*}
  \left\| \Phi(\psi_2)-\Phi(\psi_1)\right\|_{L^q_T L^r}&\lesssim
  \max_{j=1,2}\| \psi_j\|_{L^k_TL^s}^{2\si} \left\|
  \psi_2-\psi_1\right\|_{L^q_T L^r}\\
&\lesssim \left\|U_V(\cdot)\tilde \psi_-\right\|_{L^q_TL^r}^{2\si\theta
  }\|\tilde \psi_-\|_{H^1}^{2\si(1-\theta) }\left\|
  \psi_2-\psi_1\right\|_{L^q_T L^r}.
\end{align*}
Up to choosing $T$ larger, $\Phi$ is a contraction on $X_T$, equipped
with the distance
\begin{equation*}
  d(\psi_2,\psi_1) = \left\|
  \psi_2-\psi_1\right\|_{L^q_T L^r} + \left\|
  \psi_2-\psi_1\right\|_{L^\infty_T L^2},
\end{equation*}
which makes it a Banach space (see \cite{CazCourant}). 
Therefore, $\Phi$
has a unique fixed point in $X_T$, solution to
\eqref{eq:duhamel-}. It follows from \eqref{eq:A} that this solution
has indeed an $H^1$ regularity with 
\begin{equation*}
   \|\psi(t)-e^{-it H}\tilde\psi_-\|_{H^1(\R^d)}\Tend t {-
      \infty} 0.
\end{equation*}
 In view
of the global well-posedness results for the Cauchy problem associated
to \eqref{eq:psi} (see e.g. \cite{CazCourant}),
the proposition follows.
\end{proof}
\subsection{Asymptotic completeness}
\label{sec:AC-quant}

There are mainly three approaches to prove asymptotic completeness for
nonlinear Schr\"odinger equations (without potential). The initial
approach (\cite{GV79Scatt}) consists in working with a
$\Sigma$ regularity. This  
makes it possible to use the operator $x+it\nabla$, which enjoys
several nice properties, and to which an important evolution law (the
pseudo-conformal conservation law) is associated; see
Section~\ref{sec:class} for more details. This law provides
important a priori estimates, from which asymptotic completeness
follows very easily the the case $\si\ge 2/d$, and less easily for
some range of $\si$ below $2/d$; see e.g. \cite{CazCourant}. 
\smallbreak

The second historical approach relaxes the localization assumption,
and allows 
to work in $H^1(\R^d)$, provided that $\si>2/d$. It is based on
Morawetz inequalities: asymptotic completeness is then established in
\cite{LiSt78,GV85} for the case $d\ge 3$, and in \cite{NakanishiJFA} for the low
dimension cases $d=1,2$, by introducing more intricate Morawetz
estimates. Note that the case $d\le 2$ is already left out in our case, since we
have assumed $d\ge 3$ to prove Proposition~\ref{prop:waveop-quant}.  
\smallbreak

The most recent approach to prove asymptotic completeness in $H^1$
relies on the introduction of interaction Morawetz estimates in \cite{CKSTTCPAM},
an approach which has been revisited since, in particular in
\cite{PlVe09} and \cite{GiVe10}. See also \cite{Vi09} for a very nice
alternative approach of the use of interaction Morawetz estimates.  In
the presence of an external potential, this approach was used in
\cite{CaDa-p}, by working with Morrey-Campanato type norms. 
\smallbreak

An analogue for the pseudo-conformal evolution law is available (see
e.g. \cite {CazCourant}), but it seems that in the presence of $V$
satisfying Assumption~\ref{hyp:V}, it cannot be exploited to get
satisfactory estimates. We shall rather consider
Morawetz estimates as in \cite{GV85}, and thus give an alternative
proof of the corresponding result from \cite{CaDa-p}: note that for $\l=1$,
the first part of \eqref{eq:morawetz} provides exactly the same a
priori estimate as in \cite{GV85}. 
\begin{proposition}\label{prop:AC-quant}
  Let $d\ge 3$, $\frac{2}{d}<\si<\frac{2}{d-2}$, and $V$ satisfying
  Assumption~\ref{hyp:V} for some $\mu>2$. There exists $M=M(\mu,d)$ such
  that if the attractive part of the potential satisfies
  \begin{equation*}
    (\d_r V(x))_+\le \frac{M}{(1+|x|)^{\mu+1}},\quad \forall x\in
    \R^d,
  \end{equation*}
then  for
  all $\varphi\in H^1(\R^d)$, there exist a unique $\psi\in
  C(\R;H^1(\R^d))$ solution to \eqref{eq:psi} with $\psi_{\mid
    t=0}=\varphi$, and a unique $\tilde\psi_+\in
  H^1(\R^d)$ such that
  \begin{equation*}
    \|\psi(t)-e^{-itH}\tilde \psi_+\|_{H^1(\R^d)}\Tend t {+
      \infty} 0.
  \end{equation*}
In addition, $\psi,\nabla \psi\in L^q(\R_+,L^r(\R^d))$ for all
admissible pairs $(q,r)$. 
\end{proposition}
\begin{proof}
  The proof follows that argument presented in \cite{GV85} (and
  resumed in \cite{GinibreDEA}), so we shall only described the main
  steps and the modifications needed in the present context.  The key
  property in the proof consists in showing that there exists
  $2<r<\frac{2d}{d-2}$ such that 
  \begin{equation}\label{eq:Lp0}
    \|\psi(t)\|_{L^r}\Tend t {+\infty} 0.
  \end{equation}
Since $\psi\in L^\infty(\R;H^1)$ (see e.g. \cite{CazCourant}), we
infer that the above property is true for \emph{all}
$2<r<\frac{2d}{d-2}$. This aspect is the only one that requires some
adaptation in our case. Indeed, once this property is at hand, the end
of the proof relies on Strichartz estimates applied to Duhamel's
formula. In our framework, since we first want to get rid of the
nonlinearity only (and not the potential $V$ yet), we consider
\begin{equation*}
  \psi(t) = U_V(t)\varphi-i\int_0^t U_V(t-s)\(|\psi|^{2\si}\psi(s)\)ds,
\end{equation*}
and thanks to Proposition~\ref{prop:StrichartzRS}, it is possible to
follow exactly the same lines as in \cite{GV85} (see also \cite{TzVi-p})
in order to infer Proposition~\ref{prop:AC-quant}. 
\smallbreak

Therefore, the only delicate point is to show that \eqref{eq:Lp0}
holds for some $2<r<\frac{2d}{d-2}$. This corresponds to Corollary~5.1
in \cite{GV85} (Lemme~12.6 in \cite{GinibreDEA}). The main technical
remark is that once Morawetz estimate is available (the one given in
Proposition~\ref{prop:Morawetz}, whose final conclusion does not depend on
$V$), one uses dispersive properties of the group $U(t)$. As mentioned
above, we do not want to use dispersive properties of $U_V(t)$, since
they are known only in the case $d=3$ (on the other hand, this means
that the result is straightforward in the case $d=3$, from \cite{GV85}
and \cite{Go06}). So instead, we consider
Duhamel's formula for \eqref{eq:psi} in terms of $U(t)$, which reads
\begin{equation}
  \label{eq:DuhamelU}
  \psi(t)= U(t)\varphi-i\int_0^t
  U(t-s)\(|\psi|^{2\si}\psi(s)\)ds-i\int_0^t
  U(t-s)\(V\psi(s)\)ds. 
\end{equation}
The new term compared to \cite{GV85} is of course the last term in
\eqref{eq:DuhamelU}, and so the nonlinearity is now
\begin{equation*}
  f(\psi) = |\psi|^{2\si}\psi +V\psi. 
\end{equation*}
Following the argument from \cite{GV85} (or \cite{GinibreDEA}), it
suffices to prove the following two properties: \\

\noindent 1. There exist
$r_1>2^*=\frac{2d}{d-2}$ and $\alpha>0$ 
such that 
\begin{equation}
  \left\|\int_{t_0}^{t-\ell}
  U(t-s)\(V\psi(s)\)ds\right\|_{L^{r_1}(\R^d)}\le C
\ell^{-\alpha}\|\psi\|_{L^\infty(\R;H^1)}, 
\end{equation}
Consider a Lebesgue index $r_1$
slightly larger than 
$2^*$, 
\begin{equation*}
  \frac{1}{r_1} = \frac{1}{2^*} -\eta,\quad 0<\eta\ll 1. 
\end{equation*}
Let $\ell>0$, and consider
\begin{equation*}
  I_1(t) = \left\|\int_{t_0}^{t-\ell}
  U(t-s)\(V\psi(s)\)ds\right\|_{L^{r_1}(\R^d)}.
\end{equation*}
Standard dispersive estimates for $U$ yield
\begin{equation*}
  I_1(t) \lesssim \int_{t_0}^{t-\ell} (t-s)^{-\delta_1} \|V\psi(s)\|_{L^{r'_1}}ds,
\end{equation*}
where $\delta_1$ is given by
\begin{equation*}
  \delta_1 = d\(\frac{1}{2}-\frac{1}{r_1}\) = 1+\eta d.
\end{equation*}
Now we apply H\"older inequality in space, in view of the identity
\begin{equation*}
  \frac{1}{r'_1} = \frac{1}{2}+\frac{1}{d}-\eta =
  \underbrace{\frac{1}{2}-\frac{1}{d} +\eta}_{1/k} +
  \underbrace{\frac{2}{d} -2\eta}_{1/q}. 
\end{equation*}
For $\eta>0$ sufficiently small, $V\in L^q(\R^d)$ since $\mu>2$, and so
\begin{equation*}
  \|V\psi(s)\|_{L^{r'_1}} \le \|V\|_{L^{q}}\|\psi(s)\|_{L^k}\lesssim \|\psi\|_{L^\infty(\R;H^1)},
\end{equation*}
where we have used Sobolev embedding, since $2<k<2^*$. We infer
\begin{align*}
  I_1(t) &\lesssim  \int_{t_0}^{t-\ell} (t-s)^{-\delta_1} ds
  \|\psi\|_{L^\infty(\R;H^1)} \lesssim  \int_{\ell}^{\infty}
  s^{-\delta_1} ds\|\psi\|_{L^\infty(\R;H^1)}\\
& \lesssim
  \ell^{1-\delta_1}\|\psi\|_{L^\infty(\R;H^1)}=\ell^{-\eta d}\|\psi\|_{L^\infty(\R;H^1)}. 
\end{align*}

\noindent 2. Now for fixed $\ell>0$, let 
\begin{equation*}
  I_2(t) = \left\|\int_{t-\ell}^t
  U(t-s)\(V\psi(s)\)ds\right\|_{L^{2\si+2}(\R^d)}.
\end{equation*}
We show that for any  $\ell>0$, $I_2(t)\to 0$ as $t\to \infty$. 
Dispersive estimates for $U(t)$ yield
\begin{equation*}
  I_2(t) \lesssim \int_{t-\ell}^t
  (t-s)^{-\delta}\|V\psi(s)\|_{L^{\frac{2\si+2}{2\si+1}}}ds,\quad
  \delta = d\(\frac{1}{2}-\frac{1}{2\si+2}\) = \frac{d\si}{2\si+2}<1. 
\end{equation*}
For (a small) $\alpha$ to be fixed later, H\"older inequality yields
\begin{equation*}
  \|V\psi(s)\|_{L^{\frac{2\si+2}{2\si+1}}} =\left\| |x|^{\alpha}V
    \frac{\psi(s)}{|x|^\alpha} \right\|_{L^{\frac{2\si+2}{2\si+1}}}
  \le \left\| |x|^{\alpha}V \right\|_{L^{\frac{\si+1}{\si}}}
\left\| 
    \frac{\psi(s)}{|x|^\alpha} \right\|_{L^{2\si+2}}.
\end{equation*}
Note that for $0<\alpha\ll 1$, $\left\| |x|^{\alpha}V
\right\|_{L^{\frac{\si+1}{\si}}}$ is finite, since
$\frac{\si+1}{\si}>\frac{d}{2}$ and $\mu>2$. For $0<\theta<1$, write 
\begin{align*}
  \left\| 
    \frac{\psi(s)}{|x|^\alpha} \right\|_{L^{2\si+2}}=  \left\| 
    \frac{|\psi(s)|^{\theta}}{|x|^\alpha}
      |\psi(s)|^{1-\theta}\right\|_{L^{2\si+2}}&\le \left\| 
    \frac{\psi(s)}{|x|^{\alpha/\theta}}\right\|_{L^{2\si+2}}^\theta
      \left\|\psi(s)\right\|_{L^{2\si+2}}^{1-\theta}\\
   &   \lesssim  \left\| 
    \frac{\psi(s)}{|x|^{\alpha/\theta}}\right\|_{L^{2\si+2}}^\theta
      \left\|\psi\right\|_{L^\infty(\R;H^1)}^{1-\theta}.
\end{align*}
To use Morawetz estimate, we impose $\alpha/\theta= 1/(2\si+2)$, so
that we have
\begin{equation*}
   \left\| 
    \frac{\psi(s)}{|x|^\alpha} \right\|_{L^{2\si+2}} \lesssim
\(  \int_{\R^d}\frac{|\psi(s,x)|^{2\si+2}}{|x|}dx\)^{\theta/(2\si+2)}
\left\|\psi\right\|_{L^\infty(\R;H^1)}^{1-\theta}. 
\end{equation*}
We conclude by applying H\"older inequality in time: since $\delta<1$,
the map $s\mapsto (t-s)^{-\delta}$ belongs to $L^q_{\rm loc}$ for $1\le
  q\le 1+\gamma$ and $\gamma>0$ sufficiently small. Let $q=1+\gamma$
  with $0<\gamma\ll 1$ so that $s\mapsto (t-s)^{-\delta}\in L^q_{\rm
    loc}$: we have $q'<\infty$, and we can choose  $0<\theta\ll 1$ (or
  equivalently $0<\eta\ll 1$) so
  that 
  \begin{equation*}
    \theta q'=2\si+2. 
  \end{equation*}
We end up with 
\begin{equation*}
 I_2(t) \lesssim \ell^\beta \(\iint_{[t-\ell,t]\times\R^d}
   \frac{|\psi(s,x)|^{2\si+2}}{|x|}dsdx\)^{1/(2\si+2)q'},
\end{equation*}
for some $\beta>0$. The last factor goes to zero as $t\to \infty$ from
Proposition~\ref{prop:Morawetz}. 
\end{proof}

\subsection{Scattering}
\label{sec:conclusion}

Under Assumption~\ref{hyp:V}, a linear scattering theory is available,
provided that $\mu>1$; see e.g. \cite[Section~4.6]{DG}. This means that
the following strong limits exist in $L^2(\R^d)$,
\begin{equation*}
  \lim_{t\to -\infty} U_V(-t)U(t),\quad\text{and}\quad  \lim_{t\to +\infty} U(-t)U_V(t),
\end{equation*}
where the second limit usually requires to project on the continuous
spectrum. Recall that this projection is the identity in our
framework. 
\begin{lemma}\label{lem:Cook-quant}
  Let $d\ge 3$, $V$ satisfying Assumption~\ref{hyp:V} with $p>1$. Then 
the strong limit
 \begin{equation*}
  \lim_{t\to -\infty} U_V(-t)U(t)
\end{equation*}
exists in $H^1(\R^d)$. 
\end{lemma}
\begin{proof}
  Following Cook's method (\cite[Theorem~XI.4]{ReedSimon3}), it
  suffices to prove that for all $\varphi \in \Sch(\R^d)$,
  \begin{equation*}
    t\mapsto \left\| U_V(-t) VU(t)\varphi\right\|_{H^1}\in
    L^1((-\infty,-1]). 
  \end{equation*}
For the $L^2$ norm, we have
\begin{equation*}
  \left\| U_V(-t) VU(t)\varphi\right\|_{L^2} = \left\|
    VU(t)\varphi\right\|_{L^2} .
\end{equation*}
Assumption~\ref{hyp:V} implies that $V\in L^q(\R^d)$ for all
$q>d/\mu$. For $\mu>1$, let $q$ be given by 
\begin{equation*}
  \frac{1}{q} = \frac{1}{d}+\eta,\text{ with } \eta>0\text{ and }q>\frac{d}{\mu}. 
\end{equation*}
We apply H\"older inequality with the identity
\begin{equation*}
  \frac{1}{2} = \frac{1}{q} +\underbrace{\frac{1}{2}-\frac{1}{d}-\eta}_{1/r}.
\end{equation*}
Using dispersive estimates for $U(t)$, we have
\begin{equation*}
  \left\|
    VU(t)\varphi\right\|_{L^2} \lesssim \|U(t)\varphi\|_{L^r}\lesssim
  |t|^{-d\(\frac{1}{2}-\frac{1}{r}\)}\|\varphi\|_{L^{r'}}= |t|^{-1-d\eta}\|\varphi\|_{L^{r'}},
\end{equation*}
hence the existence of the strong limit in $L^2$. 
\smallbreak

For the $H^1$ limit, recall that from Lemma~\ref{lem:A}, 
\begin{equation*}
  \left\| \nabla U_V(-t) VU(t)\varphi\right\|_{L^2}\lesssim \left\| A
    U_V(-t) VU(t)\varphi\right\|_{L^2} 
\end{equation*}
Since $A$ commutes with $U_V$ which is unitary on $L^2$, the right
hand side is equal to 
\begin{equation*}
\left\| A
    VU(t)\varphi\right\|_{L^2}\lesssim \|VU(t)\varphi\|_{H^1},
\end{equation*}
where we have used Lemma~\ref{lem:A} again. Now
\begin{equation*}
  \|VU(t)\varphi\|_{H^1} \le \|VU(t)\varphi\|_{L^2}+ \|\nabla V\times
  U(t)\varphi\|_{L^2} + \|VU(t)\nabla \varphi\|_{L^2},
\end{equation*}
and each term is integrable, like for the $L^2$ limit, from
Assumption~\ref{hyp:V}. 
\end{proof}

  In the case $d=3$, the dispersive estimates established by Goldberg
  \cite{Go06} make it possible to prove asymptotic completeness in
  $H^1$ by Cook's method as well: for all $\varphi\in
  \Sch(\R^d)$,
  \begin{equation*}
    t\mapsto \|U(-t)VU_V(t)\varphi\|_{H^1}\in L^1(\R),
  \end{equation*}
a property which can be proven by the same computations as above, up
to changing the order of the arguments. To complete the proof of
Theorem~\ref{theo:scatt-quant}, it therefore remains to prove that for
$d\ge 4$, $\psi_+\in H^1(\R^d)$ and
  \begin{equation}\label{eq:cvH1}
    \|\psi(t)-U(t)\psi_+\|_{H^1(\R^d)}\Tend t \infty 0. 
  \end{equation}
It follows from the above results that
\begin{equation*}
  \psi(t) = U(t) \psi_+ +i\int_t^{+\infty}
  U(t-s)\(|\psi|^{2\si}\psi(s)\)ds
  +i\int_t^{+\infty}U(t-s)\(V(\psi(s)\)ds,
\end{equation*}
and that $\psi,\nabla \psi \in L^q(\R;L^r(\R^d))$ for all admissible
pairs $(q,r)$. Since we have
\begin{equation*}
  \psi_+= U(-t)\psi(t) - i\int_t^{+\infty}
  U(-s)\(|\psi|^{2\si}\psi(s)\)ds
 -i\int_t^{+\infty}U(-s)\(V(\psi(s)\)ds,
\end{equation*}
the previous estimates show that $\psi_+\in H^1(\R^d)$, along with
\eqref{eq:cvH1}.

\section{Scattering for the asymptotic envelope}
\label{sec:class}

In this section, we prove Theorem~\ref{theo:scatt-class}. The general
argument is similar to the quantum case: we first prove that the
nonlinear term can be neglected to large time, and then rely on
previous results to neglect the potential. 
Recall that in view of Assumption~\ref{hyp:V}, the time dependent
harmonic potential $\frac{1}{2}\<Q(t)y,y\>$ satisfies
 \begin{equation}\label{eq:decayQ}
   \left\|\frac{d^\alpha}{dt^\alpha}Q(t)\right\|\lesssim
   \<t\>^{-\mu-2-\alpha},\quad \alpha \in \N,
 \end{equation}
where $\|\cdot\|$ denotes any matricial norm. 
We denote by 
\begin{equation*}
  H_Q = -\frac{1}{2}\Delta + \frac{1}{2}\<Q(t)y,y\>
\end{equation*}
the time-dependent Hamiltonian present in \eqref{eq:u}. Like in the
quantum case, we show that the nonlinearity is negligible for large
time by working on Duhamel's formula associated to \eqref{eq:u} in
terms of $H_Q$. Since $H_Q$ depends on time, we recall that the
propagator $U_Q(t,s)$ is the operator which maps $u_0$ to $u_{\rm lin}(t)$,
where $u_{\rm lin}$ solves
\begin{equation*}
  i\d_t u_{\rm lin} +\frac{1}{2}\Delta u_{\rm lin} =
  \frac{1}{2}\<Q(t)y,y\>u_{\rm lin};\quad u_{{\rm lin}}(s,y)=u_0(y). 
\end{equation*}
It is a unitary dynamics, in the sense that $U_Q(s,s)=1$, and
$U_Q(t,\tau)U_Q(\tau,s)=U_Q(t,s)$;
see e.g. \cite{DG}. Then to prove the existence of wave operators, we consider the
integral formulation
\begin{equation}
  \label{eq:duhamel-wave-class}
  u(t) = U_Q(t,0)\tilde u_--i\int_{-\infty}^t U_Q(t,s)\(|u|^{2\si}u(s)\)ds.
\end{equation}
A convenient tool is given by Strichartz estimates associated to
$U_Q$. Local in time Strichartz estimates follow from general results
given in \cite{Fujiwara}, where local dispersive estimates are
proven for more general potential. To address large time,  we take
advantage of the fact that the 
potential is exactly quadratic with respect to the space variable, so
an explicit formula is available for $U_Q$, entering the general
family of Mehler's formulas (see e.g. \cite{Feyn,HormanderQuad}). 
\subsection{Mehler's formula}
\label{sec:mehler}

Consider, for $t_0\ll -1$,
\begin{equation*}
i\d_tu+\frac{1}{2}\Delta u=\frac{1}{2}\< Q(t)y,y\> u\quad
;\quad u(t_0,y)=u_0(y).
\end{equation*}
We seek a solution of the form
\begin{equation}
  \label{eq:mehler}
  u(t,y) = \frac{1}{h(t)}\int_{\R^d}
  e^{\frac{i}{2}\(\<M_1(t)y,y\>+\<M_2(t)z,z\>+2\<P(t)y,z\>\)}u_0(z)dz, 
\end{equation}
with symmetric matrices $M_1, M_2,P\in \mathcal S_d(\R)$. 
Experience shows that no linear term is needed in this formula, since
the potential is exactly quadratic (see
e.g. \cite{CLSS08}). 
\smallbreak

We compute:
\begin{align*}
  i\d_t u & = -i\frac{\dot h}{h}u -\frac{1}{2}\<\dot M_1(t)y,y\>u\\
&\quad 
  +\frac{1}{h}\int e^{\frac{i}{2}\(\dots\)} \(-\frac{1}{2}\<\dot
  M_2(t)z,z\>-\<\dot P(t)y,z\>\)u_0(z)dz,
\end{align*}
\begin{align*}
  \d_{j}^2 u &= \frac{1}{h}\int e^{\frac{i}{2}\(\dots\)}
  \(-\(\(M_1(t)y\)_j + \(P(t)z\)_j\)^2 -i\(M_1\)_{jj}\)u_0(z)dz,
\end{align*}
hence
\begin{align*}
 & i\d_tu+\frac{1}{2}\Delta u =  -i\frac{\dot h}{h}u
  +\frac{i}{2}\operatorname{tr} M_1 - \frac{1}{2}\<\dot M_1(t)y,y\>u\\
&+ \frac{1}{2h}\int
  e^{\frac{i}{2}\(\<M_1(t)y,y\>+\<M_2(t)z,z\>+2\<P(t)y,z\>\)}u_0(z)\times\\
&\times\( 
-\<\dot  M_2(t)z,z\>-2\<\dot
P(t)y,z\>-|M_1(t)y|^2 -|P(t)z|^2 -2
\<M_1(t)y,P(t)z\>\)dz. 
\end{align*}
Identifying the quadratic forms (recall that the matrices $M_j$ and
$P$ are symmetric), we find:
\begin{align*}
&\frac{\dot h}{h}=  \frac{1}{2}\operatorname{tr} M_1,\\
& \dot M_1+M_1^2+Q=0,\\
& \dot M_2 +P^2=0,\\
& \dot P + PM_1=0.
\end{align*}
Dispersion is given by
\begin{equation*}
  h(t)  = h(t_1)\exp\(\frac{1}{2}\int_{t_1}^t \operatorname{tr} M_1(s)ds\),
\end{equation*}
where $M_1$ solves the matrix Riccati equation
\begin{equation}
  \label{eq:riccati}
   \dot M_1 + M_1^2 + Q=0;\quad M_1(t_0)=\frac{1}{t_0}{\rm I}_d.
\end{equation}
Note that in general, solutions to Riccati equations develop
singularities in finite time. What saves the day here is that
\eqref{eq:riccati} is not translation invariant, and can be
considered, for $t\le t_0\ll -1$, 
as a perturbation of the Cauchy problem
\begin{equation*}
  \dot M + M^2 =0;\quad M(t_0)=\frac{1}{t_0}{\rm I}_d,
\end{equation*}
whose solution is given by 
\begin{equation*}
  M(t) = \frac{1}{t}{\rm I}_d. 
\end{equation*}
\begin{lemma}\label{lem:riccati}
  Let $Q$ be a symmetric matrix satisfying \eqref{eq:decayQ} for $\mu>1$. There
  exists $t_0<0$ such that \eqref{eq:riccati} has a unique solution
  $M_1\in C((-\infty,t_0];\mathcal S_d(\R))$. In addition, it
  satisfies
  \begin{equation*}
    M_1(t)= \frac{1}{t}{\rm I}_d +\O\(\frac{1}{t^2}\)\quad \text{as
    }t\to -\infty. 
  \end{equation*}
\end{lemma}
\begin{proof}
Seek a solution of the form $M_1(t) =   \frac{1}{t}{\rm I}_d +R(t)$,
where $R$ is s symmetric matrix solution of
\begin{equation*}
  \dot R + \frac{2}{t}R+R^2 +Q= 0;\quad R(t_0)=0. 
\end{equation*}
Equivalently, the new unknown $\tilde R = t^2 R$ must satisfy
\begin{equation}\label{eq:Rmatrix}
   \dot {\tilde R} + \frac{1}{t^2}\tilde R^2 +t^2Q= 0;\quad \tilde R(t_0)=0. 
\end{equation}
Cauchy-Lipschitz Theorem yields a local solution: we show that it is
defined on $(-\infty,t_0]$, along with the announced decay. 
Integrating between $t_0$ and $t$, we find
\begin{equation*}
  \tilde R(t) = -\int_{t_0}^t \frac{1}{s^2}\tilde R(s)^2ds -
  \int_{t_0}^ts^2 Q(s)ds.
\end{equation*}
Note that $s\mapsto s^2 Q$ is integrable as $s\to -\infty$ from \eqref{eq:decayQ}
(we assume $\mu>1$). Setting 
\begin{equation*}
  \rho(t) =\sup_{t\le s\le t_0}\|\tilde R(s)\|,
\end{equation*}
where $\|\cdot\|$ denotes any matricial norm, we have
\begin{equation*}
  \rho(t) \le \frac{C}{t_0}\rho(t)^2 + \frac{C}{t_0^{\mu-1}},
\end{equation*}
for some constant $C$. Choosing $t_0\ll -1$, global existence follows
from the following bootstrap argument (see \cite{BG3}):
Let $f=f(t)$ be a nonnegative continuous function on $[0,T]$ such
that, for every $t\in [0,T]$, 
\begin{equation*}
  f(t)\le \eps_1 + \eps_2 f(t)^\theta,
\end{equation*}
where $\eps_1,\eps_2>0$ and $\theta >1$ are constants such that
\begin{equation*}
  \eps_1 <\left(1-\frac{1}{\theta} \right)\frac{1}{(\theta \eps_2)^{1/(\theta
-1)}}\ ,\ \ \ f(0)\le  \frac{1}{(\theta \eps_2)^{1/(\theta-1)}}.
\end{equation*}
Then, for every $t\in [0,T]$, we have
\begin{equation*}
  f(t)\le \frac{\theta}{\theta -1}\ \eps_1.
\end{equation*}
This shows that for $|t_0|$ sufficiently large, the matrix $R$ (hence
$M_1$) is defined on $(-\infty,t_0]$. Moreover, since $\tilde R$ is
bounded, $R(t)=\O(t^{-2})$ as $t\to -\infty$, hence the result. 
\end{proof}
We infer
\begin{equation*}
  h(t)\Eq t {-\infty} c|t|^{d/2},
\end{equation*}
which is the same dispersion as in the case without
potential. Putting this result together with local dispersive estimates from
\cite{Fujiwara}, we have:
\begin{lemma}\label{lem:strichartz-quad}
  Let $Q$ be a symmetric matrix satisfying \eqref{eq:decayQ} for
  $\mu>1$. Then for all admissible pairs $(q,r)$, 
there exists $C=C(q,d)$ such that for all $s\in \R$,
\begin{equation*}
  \|U_Q(\cdot,s)f\|_{L^q(\R;L^r(\R^d))}\le C \|f\|_{L^2(\R^d)},\quad
  \forall f\in L^2(\R^d). 
\end{equation*}
For two admissible pairs $(q_1,r_1)$ and $(q_2,r_2)$, there exists $C_{q_1,q_2}$ such
that for all time interval $I$, if we denote by
\begin{equation*}
  R(F)(t,y) = \int_{I\cap \{s\le t\}} U_Q(t,s)F(s,y)ds,
\end{equation*}
we have
\begin{equation*}
  \|R(F)\|_{L^{q_1}(I;L^{r_1}(\R^d))}\le
  C_{q_1,q_2}\|F\|_{L^{q_2'}(I;L^{r_2'}(\R^d))},\quad \forall F\in
  L^{q_2'}(I;L^{r_2'}(\R^d)). 
\end{equation*}
\end{lemma}
\begin{remark}
  Since we have dispersive estimates, end-point Strichartz estimates
  ($q=2$ when $d\ge 3$)
  are also available from \cite{KT}. 
\end{remark}
\subsection{Wave operators}
\label{sec:wave-class}

In this section, we prove:
\begin{proposition}\label{prop:wave-class}
   Let $d\ge 1$, $\frac{2}{d}\le \si<\frac{2}{(d-2)_+}$, and $V$ satisfying
  Assumption~\ref{hyp:V} for some $\mu>1$. For
  all $\tilde u_-\in \Sigma$, there exists a unique $u\in
  C(\R;\Sigma)$ solution to \eqref{eq:u}  such that
  \begin{equation*}
    \|U_Q(0,t)u(t)-\tilde u_-\|_{\Sigma}\Tend t {-
      \infty} 0.
  \end{equation*}
\end{proposition}
\begin{remark}
  The assumption $\si\ge \frac{2}{d}$ could easily be relaxed,
  following the classical argument (see e.g. \cite{CazCourant}). We do
  not present the argument, since Theorem~\ref{theo:scatt-quant} is
  proven only for $\si>\frac{2}{d}$. 
\end{remark}

\begin{proof}
  The proof follows closely the approach without potential
  ($Q=0$). From this perspective, a key tool is the vector field
  \begin{equation*}
    J(t)=y+it\nabla.
  \end{equation*}
It satisfies three important properties:
\begin{itemize}
\item It commutes with the free Schr\"odinger dynamics,
  \begin{equation*}
   \left[ i\d_t +\frac{1}{2}\Delta,J\right]=0.  
  \end{equation*}
\item It acts like a derivative on gauge invariant nonlinearities. If
  $F(z)$ is of the form $F(z)=G(|z|^2)z$, then 
  \begin{equation*}
    J(t)\(F(u)\) = \d_z F(u)J(t)u -\d_{\bar z}F(u)\overline{J(t)u}.
  \end{equation*}
\item It provides weighted Gagliardo-Nirenberg inequalities:
  \begin{align*}
    \|f\|_{L^r}\lesssim &
    \frac{1}{|t|^{\delta(r)}}\|f\|_{L^2}^{1-\delta(r)}\|J(t)f\|_{L^2}^{\delta(r)},
    \quad \delta(r)=d\(\frac{1}{2}-\frac{1}{r}\), \\
&\text{with }
\left\{
  \begin{aligned}
    2\le r\le \infty &\text{ if }d=1,\\
2\le r<\infty &\text{ if }d=2,\\
2\le r\le \frac{2d}{d-2}&\text{ if }d\ge 3. 
  \end{aligned}
\right.
  \end{align*}
\end{itemize}
The last two properties stem from the factorization $J(t)f =
it e^{i\frac{|y|^2}{2t}}\nabla \(e^{-i\frac{|y|^2}{2t}}f\)$. Note that
the commutation property does not incorporate the quadratic potential:
\begin{align*}
  \left[ i\d_t -H_Q,J\right]= itQ(t)y=itQ(t)J(t) +t^2 Q(t)\nabla. 
\end{align*}
Now the important remark is that $t\mapsto t^2Q(t)$ is integrable,
from \eqref{eq:decayQ} since $\mu>1$.  
\smallbreak

To prove Proposition~\ref{prop:wave-class}, we apply a fixed point argument 
to the Duhamel's formula \eqref{eq:duhamel-wave-class}. As in the case
of the quantum scattering operator, we have to deal with the fact that
the gradient does not commute with $U_Q$, leading to the problem
described in Section~\ref{sec:vector-field}. Above, we have sketched
how to deal with the inhomogeneous term in
\eqref{eq:duhamel-wave-class}, while in
Section~\ref{sec:vector-field}, we had underscored the difficulty
related to the homogeneous term. We therefore start by showing that
for any admissible pair $(q_1,r_1)$, there exists $K_{q_1}$ such that
\begin{equation}\label{eq:Qhomo}
  \|\nabla U_Q(t,0)f\|_{L^{q_1}(\R;L^{r_1})} +
  \|J(t)U_Q(t,0)f\|_{L^{q_1}(\R;L^{r_1})} \le K_{q_1} \|f\|_{\Sigma}. 
\end{equation}
To prove this, denote 
\begin{equation*}
  v_0(t)=U_Q(t,0)f,\quad v_1(t)= \nabla U_Q(t,0)f, \quad v_2(t)=
J(t)U_Q(t,0)f.
\end{equation*}
Since $yv_0 =v_2-it v_1$, we have:
\begin{align*}
  &i\d_t v_1=H_Q v_1 +Q(t)yv_0 = Hv_1 +Q(t)v_2-it Q(t)v_1;\quad
    v_1(0,y)=\nabla f(y),\\
& i\d_t v_2 = H_Qv_2 +itQ(t)v_2+t^2Q(t)v_1;\quad v_2(0,y)=yf(y). 
\end{align*}
Lemma~\ref{lem:strichartz-quad} yields
\begin{align*}
  \|v_1\|_{L^{q_1}(\R;L^{r_1})} + \|v_2\|_{L^{q_1}(\R;L^{r_1})}
  &\lesssim \|f\|_\Sigma + \int_{-\infty}^\infty \|\<t\>Q(t)v_2(t)\|_{L^2}dt \\
&\quad +
  \int_{-\infty}^\infty \|\<t\>^2Q(t)v_1(t)\|_{L^2}dt ,
\end{align*}
where we have chosen $(q_2,r_2)=(\infty,2)$. The fact that $U_Q$ is
unitary on $L^2$ and \eqref{eq:decayQ} imply
\begin{equation*}
  \|\<t\>Q(t)v_2(t)\|_{L^2}\lesssim \<t\>^{-\mu-1}\|yf\|_{L^2},\quad
  \|\<t\>^2Q(t)v_1(t)\|_{L^2}\lesssim \<t\>^{-\mu}\|\nabla f\|_{L^2}, 
\end{equation*}
hence \eqref{eq:Qhomo}. 
We then apply a fixed point
argument in
\begin{align*}
  X(T) =&\Big\{ u\in L^\infty((-\infty,-T];H^1),  \\
& \quad\sum_{B\in \{{\rm Id}, \nabla, J\}}
\( \|B   u\|_{L^\infty((-\infty,-T];L^2)}+
\|B  u\|_{L^q((-\infty,-T];L^r)}\)\le {\mathbf K}\|\tilde u_-\|_{\Sigma}\Big\},
\end{align*}
where the admissible pair $(q,r)$ is given by
\begin{equation*}
  (q,r) = \(\frac{4\si+4}{d\si},2\si+2\),
\end{equation*}
 and the constant $\mathbf K$ is related to the constants $C_q$  from
 Strichartz inequalities 
(Lemma~\ref{lem:strichartz-quad}), and $K_q$  from
\eqref{eq:Qhomo}, whose value we do not try to optimize. The fixed
point argument is applied 
to the Duhamel's formula \eqref{eq:duhamel-wave-class}: we denote by
$\Phi(u)$ the left hand side, and let $u\in X(T)$. We have 
\begin{equation*}
  \|\Phi(u)\|_{L^\infty((-\infty,-T];L^2)}\le \|\tilde u_-\|_{L^2} + C
    \left\| |u|^{2\si}u\right\|_{L^{q'}_TL^{r'}},
\end{equation*}
where $L^a_T$ stands for $L^a((-\infty,-T])$. H\"older inequality
yields
\begin{equation*}
  \left\| |u|^{2\si}u\right\|_{L^{q'}_TL^{r'}} \le
  \|u\|_{L^k_TL^r}^{2\si} \|u\|_{L^q_TL^r},
\end{equation*}
where $k$ is given by
\begin{equation*}
  \frac{1}{q'}=\frac{1}{q}+\frac{2\si}{k},\text{ that is }k =
  \frac{4\si(\si+1)}{2-(d-2)\si}. 
\end{equation*}
Weighted Gagliardo-Nirenberg inequality and the definition of $X(T)$ yield
\begin{equation*}
  \|u(t)\|_{L^r}\lesssim \frac{1}{|t|^{\frac{d\si}{2\si+2}}}\|u_-\|_{\Sigma}.
\end{equation*}
We check that for $\si\ge \frac{2}{d}$, 
\begin{equation*}
  k\times \frac{d\si}{2\si+2} = \frac{2d\si^2}{2-(d-2)\si}\ge 2,
\end{equation*}
and so
\begin{equation*}
  \|u\|_{L^k_TL^r}^k =\O\(\frac{1}{T}\) \text{ as }T\to \infty. 
\end{equation*}
By using Strichartz estimates again,
\begin{equation*}
  \|\Phi(u)\|_{L^q_TL^r}\le C_q\|\tilde u_-\|_{L^2} + C
    \left\| |u|^{2\si}u\right\|_{L^{q'}_TL^{r'}},
\end{equation*}
which shows, like above, that if $T$ is sufficiently large,
$\|\Phi(u)\|_{L^q_TL^r}\le 2C_q\|\tilde u_-\|_{L^2}$. 
\smallbreak

We now apply $\nabla$ and $J(t)$ to $\Phi$, and get a closed system of
estimates:
\begin{align*}
  \nabla \Phi(u) &= \nabla U_Q(t,0)\tilde u_- - i \int_{-\infty}^t
  U_Q(t,s)\nabla \(|u|^{2\si}u(s)\)ds \\
& -i\int_{-\infty}^t U_Q(t,s)\(Q(s)J(s)\Phi(u)\)ds - \int_{-\infty}^t
  U_Q(t,s)\(sQ(s)\nabla\Phi(u)\)ds, \\
 J(t) \Phi(u) &= J(t) U_Q(t,0)\tilde u_- - i \int_{-\infty}^t
  U_Q(t,s)J(s) \(|u|^{2\si}u(s)\)ds \\
& +\int_{-\infty}^t U_Q(t,s)\(sQ(s)J(s)\Phi(u)\)ds - i\int_{-\infty}^t
  U_Q(t,s)\(s^2Q(s)\nabla\Phi(u)\)ds,
\end{align*}
where we have used the same algebraic properties as in the proof of
\eqref{eq:Qhomo}. Set 
\begin{equation*}
  M(T) = \sum_{B\in \{\nabla, J\}} \( \|B(t)\Phi(u)\|_{L^\infty_TL^2}
  + \|B(t)\Phi(u)\|_{L^q_TL^r}\).
\end{equation*}
Lemma~\ref{lem:strichartz-quad} and
\eqref{eq:Qhomo} yield
\begin{align*}
  M(T)&\lesssim \|\tilde u_-\|_\Sigma + \sum_{B\in \{\nabla, J\}}\left\||u|^{2\si}B
        u\right\|_{L^{q'}_TL^{r'}} \\
&\quad+\|\<t\>Q(t)J(t)\Phi(u)\|_{L^1_TL^2} +
        \|\<t\>^2Q(t)\nabla\Phi(u)\|_{L^1_TL^2}, 
\end{align*}
where we have also used the fact that $J(t)$ acts like a derivative on
gauge invariant nonlinearities. The same H\"older inequalities as
above yield
\begin{equation*}
  \left\||u|^{2\si}B
        u\right\|_{L^{q'}_TL^{r'}} \le
      \|u\|_{L^k_TL^r}^{2\si}\|Bu\|_{L^q_TL^r}\lesssim \frac{1}{T^{2\si/k}}\|Bu\|_{L^q_TL^r}.
\end{equation*}
On the other hand, from \eqref{eq:decayQ},
\begin{equation*}
  \|\<t\>Q(t)J(t)\Phi(u)\|_{L^1_TL^2} +
        \|\<t\>^2Q(t)\nabla\Phi(u)\|_{L^1_TL^2}\lesssim \frac{1}{T^{\mu-1}}M(T),
\end{equation*}
and so
\begin{equation*}
  M(T) \lesssim \|\tilde u_-\|_\Sigma  +
  \frac{1}{T^{2\si/k}}\sum_{B\in \{\nabla, J\}} \|Bu\|_{L^q_TL^r} + \frac{1}{T^{\mu-1}}M(T).
\end{equation*}
By choosing $T$ sufficiently large, we infer
\begin{equation*}
  M(T) \lesssim \|\tilde u_-\|_\Sigma  +
  \frac{1}{T^{2\si/k}}\sum_{B\in \{\nabla, J\}} \|Bu\|_{L^q_TL^r},
\end{equation*}
and we conclude that $\Phi$ maps $X(T)$ to $X(T)$ for $T$
sufficiently large. Up to choosing $T$ even larger, $\Phi$ is a
contraction on $X(T)$ with respect to the weaker norm $L^q_TL^r$,
since for $u,v\in X(T)$, we have 
\begin{align*}
  \|\Phi(u)-\Phi(v)\|_{L^q_TL^r}&\lesssim \left\| |u|^{2\si}u
    -|v|^{2\si}v\right\|_{L^{q'}_TL^{r'}}\lesssim \(
  \|u\|_{L^k_TL^r}^{2\si}+\|v\|_{L^k_TL^r}^{2\si}\)\|u-v\|_{L^q_TL^r}\\
&\lesssim \frac{1}{T^{2\si/k}}\|u-v\|_{L^q_TL^r},
\end{align*}
where we have used the previous estimate. Therefore, there exists
$T>0$ such that $\Phi$ has a unique fixed point in $X(T)$. This
solution actually belongs to $C(\R;\Sigma)$ from \cite{CaSi15}.  
Unconditional uniqueness (in $\Sigma$, without referring to mixed
space-time norms) stems from the approach in \cite{TzVi-p}. 
\end{proof}

\subsection{Vector field}
\label{sec:vector-field-Q}

  It is possible to construct a vector field adapted to the presence
  of $Q$, even though it is not needed to prove
  Proposition~\ref{prop:wave-class}. Such a vector field will be useful
  in Section~\ref{sec:cv}, and since its construction is very much in
  the continuity of Section~\ref{sec:mehler}, we present it now. Set,
  for a scalar function $f$, 
\begin{equation*}
  {\mathcal A} f= i W(t) e^{i\phi(t,y)}\nabla \(
    e^{-i\phi(t,y)}f\)= W(t) \(f\nabla \phi +i\nabla f\),
\end{equation*}
where $W$ is a matrix and the phase $\phi$ solves the eikonal equation
\begin{equation*}
 \d_t \phi +\frac{1}{2}|\nabla \phi|^2 + \frac{1}{2}\< Q(t)y,y\>=0. 
\end{equation*}
Since the underlying Hamiltonian is quadratic, $\phi$ has the form
\begin{equation*}
  \phi(t,y) = \frac{1}{2}\<K(t)y,y\>,
\end{equation*}
where $K(t)$ is a symmetric matrix. For $ {\mathcal A}$ to commute
with $i\d_t -H_Q$, we come up with the conditions
\begin{equation*}
   \dot K + K^2 + Q=0,\quad  \dot W = W \nabla^2\phi= WK .
\end{equation*}
We see that we can take $K=M_1$ as in the proof of
Lemma~\ref{lem:riccati}, and $ {\mathcal A}$ will then satisfy the
same three properties as $J$, up to the fact that the commutation
property now includes the quadratic potential. 
\smallbreak

Since the construction of this vector field
boils down to solving a matricial Riccati equation with initial data
prescribed at large time (see \eqref{eq:riccati}), we naturally
construct two vector fields $\mathcal A_\pm$, associated to $t\to \pm
\infty$. In view of Lemma~\ref{lem:riccati}, $\mathcal A_-$ is defined
on $(-\infty,-T]$, while $\mathcal A_+$ is defined
on $[T,\infty)$, for a common $T\gg 1$, with
\begin{equation*}
  \mathcal A_\pm  = W_\pm(t)\(\nabla \phi_\pm + i\nabla\), \quad
  \phi_\pm (t,y) = \frac{1}{2}\<K_\pm (t)y,y\>,
\end{equation*}
where $K_\pm$ and $W_\pm$ satisfy
\begin{equation*}
  \dot K_\pm +K_\pm^2+Q=0,\quad \dot W_\pm =W_\pm K_\pm,
\end{equation*}
so that Lemma~\ref{lem:riccati} also yields
\begin{equation}\label{eq:vector-asym}
  K_\pm(t)\sim  \frac{1}{t}{\rm I}_d,\quad  W_\pm (t)\sim t {\rm
    I}_d\quad \text{as }t\to \pm \infty.  
\end{equation} 
We construct commuting vector fields for large time only, essentially
because on finite time intervals, the absence of commutation is not a
problem, so we can use $\nabla$, $y$ or $J$.

\subsection{Asymptotic completeness}
\label{sec:ac-class}

In this section we prove:
\begin{proposition}\label{prop:AC-class}
  Let $d\ge 1$, $\frac{2}{d}\le \si<\frac{2}{(d-2)_+}$, and $V$
  satisfying Assumption~\ref{hyp:V} for some $\mu>1$. For all $u_0\in
  \Sigma$, there exists a unique $\tilde u_+\in \Sigma$ such that the
  solution $u\in C(\R;\Sigma)$ to \eqref{eq:u} with $u_{\mid t=0}=u_0$
  satisfies
  \begin{equation*}
   \sum_{\Gamma\in \{{\rm Id},\nabla, J\}} \|\Gamma(t) u(t)-
   \Gamma(t)U_Q(t,0)\tilde u_+\|_{L^2} \Tend t {+\infty} 0. 
  \end{equation*}
\end{proposition}
\begin{proof}
  In the case $Q=0$, such a result is a rather direct consequence of
  the \emph{pseudo-conformal conservation law}, established in
  \cite{GV79Scatt}. Recalling that $J(t)=y+it\nabla$, this law reads
\begin{equation*}
  \frac{d}{dt}\(\frac{1}{2}\|J(t)u\|_{L^2}^2
  +\frac{t^2}{\si+1}\|u(t)\|_{L^{2\si+2}}^{2\si+2}\)
  =\frac{t}{\si+1}(2-d\si)\|u(t)\|_{L^{2\si+2}}^{2\si+2}. 
\end{equation*}
A way to derive this relation is to apply $J$ to \eqref{eq:u}. The operator $J$
commutes with the linear part ($Q=0$), and the standard $L^2$
estimate, which consists in multiplying the outcome by
$\overline{Ju}$, integrating in space, and taking the imaginary part, yields:
\begin{equation*}
\frac{1}{2} \frac{d}{dt}\|J(t)u\|_{L^2}^2  = \IM \int \overline {Ju}J\(|u|^{2\si}u\).
\end{equation*}
Since we have $J= i t  e^{i\frac{|y|^2}{2t}} \nabla\( \cdot
e^{-i\frac{|y|^2}{2t}}\)$, 
\begin{equation*}
  J\(|u|^{2\si}u\) = (\si+1)|u|^{2\si}Ju + \si u^{\si+1}\bar
  u^{\si-1}\overline{Ju}.
\end{equation*}
The first term is real, and the rest of the computation consists in
expanding the remaining term. 
\smallbreak

In the case where $Q\not =0$, we resume the above approach: the new
contribution is due to the fact that $J$ does not commute with the
external potential, so we find:
\begin{align*}
  \frac{1}{2} \frac{d}{dt}\|J(t)u\|_{L^2}^2 & =\text{like before} + \RE \int
  t Q(t)xu\cdot \overline {Ju}\\
&=\text{like before} + 
  t\RE\int_{\R^d} \<Q(t) J(t)u,J(t)u\> +t^2 \IM \int_{\R^d} \<Q(t)\nabla u,Ju\>.
\end{align*}
On the other hand, we still have
\begin{align*}
  \frac{d}{dt}\|u(t)\|_{L^{2\si+2}}^{2\si+2}& =2 (\si+1)\int
  |u|^{2\si}\RE \(\bar u\d_tu\) = 2 (\si+1)\int
  |u|^{2\si}\RE \(\bar u \times\frac{i}{2}\Delta u\) ,
\end{align*}
and so,
\begin{align*}
  \frac{d}{dt}\(\frac{1}{2}\|J(t)u\|_{L^2}^2
  +\frac{t^2}{\si+1}\|u(t)\|_{L^{2\si+2}}^{2\si+2}\)
  &=\frac{t}{\si+1}(2-d\si)\|u(t)\|_{L^{2\si+2}}^{2\si+2}\\
+ 
 t\RE\int_{\R^d} \<Q(t) J(t)u,J(t)u\>& +t^2 \IM \int_{\R^d} \<Q(t)\nabla u,Ju\> . 
\end{align*}
Thus for $t\ge 0$ and $\si\ge\frac{2}{d}$, \eqref{eq:decayQ} implies
\begin{equation*}
  \frac{d}{dt}\(\frac{1}{2}\|J(t)u\|_{L^2}^2
  +\frac{t^2}{\si+1}\|u(t)\|_{L^{2\si+2}}^{2\si+2}\)
  \lesssim  
  \<t\>^{-\mu-1}\|J(t)u\|_{L^2}^2 +\<t\>^{-\mu}\| \nabla u\|_{L^2}  \|Ju\|_{L^2}. 
\end{equation*}
Even though there is no conservation of the energy for \eqref{eq:u}
since the potential depends on time, we know from \cite{Ha13} that $u\in
L^\infty(\R;H^1(\R^d))$. As a matter of fact, the proof given in
\cite[Section~4]{Ha13} concerns the case $\si=1$ in $d=2$ or $3$, but
the argument, based on energy estimates,  remains valid for $d\ge 1$,
$\si<\frac{2}{(d-2)_+}$, since we then know that $u\in
C(\R;\Sigma)$. Since $\mu>1$, we infer
\begin{equation}\label{eq:Jborne}
  Ju\in L^\infty(\R_+;L^2). 
\end{equation}
Writing Duhamel's formula for \eqref{eq:u} with initial datum $u_0$,
in terms of $U_Q$, we have
\begin{equation*}
  u(t) = U_Q(t,0)u_0-i\int_0^t U_Q(t,s)\(|u|^{2\si}u(s)\)ds.
\end{equation*}
Resuming the computations presented in the proof of
Proposition~\ref{prop:wave-class}, \eqref{eq:Jborne} and (weighted)
Gagliardo-Nirenberg inequalities make it possible to prove that
\begin{equation*}
  Bu \in L^{q_1}(\R_+;L^{r_1}),\ \forall (q_1,r_1)\text{ admissible},
  \ \forall B\in \{{\rm Id},\nabla, J\}. 
\end{equation*}
Duhamel's formula then yields, for $0<t_1<t_2$,
\begin{equation*}
  U_Q(0,t_2)u(t_2)-U_Q(0,t_1)u(t_1) =
  -i\int_{t_1}^{t_2}U_Q(0,s)\(|u|^{2\si}u(s)\)ds. 
\end{equation*}
From Strichartz estimates,
\begin{equation*}
  \| U_Q(0,t_2)u(t_2)-U_Q(0,t_1)u(t_1) \|_{L^2}\lesssim \left\|
    |u|^{2\si}u\right\|_{L^{q'}([t_1,t_2]:L^{r'})}, 
\end{equation*}
and the right hand side goes to zero as $t_1,t_2\to
+\infty$. Therefore, there exists (a unique) $\tilde u_+\in L^2$ such that
\begin{equation*}
   \| U_Q(0,t)u(t)-\tilde u_+ \|_{L^2}\Tend t {+\infty} 0 ,
\end{equation*}
and we have
\begin{equation*}
  u(t) = U_Q(t,0)\tilde u_+ +i\int_t^\infty
  U_Q(t,s)\(|u|^{2\si}u(s)\)ds. 
\end{equation*}
Using the same estimates as in the proof of
Proposition~\ref{prop:wave-class}, we infer
\begin{align*}
  \|\nabla u(t) - \nabla  U_Q(t,0)\tilde u_+\|_{L^2} &+  \|J(t) u(t) -
  J(t)  U_Q(t,0)\tilde u_+\|_{L^2} \\
& \lesssim \left\| |u|^{2\si}\nabla
  u\right\|_{L^{q'}(t,\infty;L^{r'})} +  \left\| |u|^{2\si} J
  u\right\|_{L^{q'}(t,\infty;L^{r'})}  \\
&\quad+
  \|\<s\>^{-\mu-1}J(s)u\|_{L^1(t,\infty;L^2)} +  \|\<s\>^{-\mu}\nabla u\|_{L^1(t,\infty;L^2)}. 
\end{align*}
The right hand side goes to zero as $t\to \infty$, hence the
proposition. 
\end{proof}

\begin{remark}
  As pointed out in the previous section, it would be possible to
  prove the existence of wave operators by using an adapted vector
  field $\mathcal A$. On the other hand, if $Q(t)$ is not proportional
  to the identity matrix, it seems that no (exploitable) analogue of
  the pseudo-conformal conservation law is available in terms of
  $\mathcal A$ rather than in terms of $J$. 
\end{remark}
\subsection{Conclusion}
\label{sec:concl-class}

Like in the case of quantum scattering, we use a stronger version of
the linear scattering theory:
\begin{proposition}\label{prop:Cook-class}
  Let $d\ge 1$, $V$ satisfying Assumption~\ref{hyp:V} with $\mu>1$. Then 
the strong limits
 \begin{equation*}
  \lim_{t\to \pm \infty} U_Q(0,t)U(t) \quad \text{and}\quad  \lim_{t\to
    \pm\infty} U(-t) U_Q(t,0) \quad \text{and}\quad  
\end{equation*}
exist in $\Sigma$. 
\end{proposition}
\begin{proof}
  For the first limit (existence of wave operators), again in view of
  Cook's method, we prove that for all $\varphi\in 
  \Sch(\R^d)$, 
\begin{equation*}
    t\mapsto \left\| U_Q(0,t) \<Q(t)y,y\>U(t)\varphi\right\|_{\Sigma}\in
    L^1(\R). 
  \end{equation*}
For the $L^2$ norm, we have, in view of \eqref{eq:decayQ},
\begin{equation*}
  \left\| U_Q(0,t) \<Q(t)y,y\>U(t)\varphi\right\|_{L^2} \lesssim
  \<t\>^{-\mu-2}\sum_{j=1}^d\| y_j^2 U(t)\varphi\|_{L^2}.
\end{equation*}
Write
\begin{equation*}
  y_j^2 = (y_j+it\d_j)^2 +t^2\d_j^2 -2ity_j\d_j =  (y_j+it\d_j)^2
  -t^2\d_j^2 -2it(y_j+it\d_j)\d_j,
\end{equation*}
to take advantage of the commutation
\begin{equation*}
  (y_j+it\d_j)U(t) = U(t)y_j,
\end{equation*}
and infer
\begin{equation*}
  \left\| U_Q(0,t) \<Q(t)y,y\>U(t)\varphi\right\|_{L^2} \lesssim
  \<t\>^{-\mu-2}\(\||y|^2\varphi\|_{L^2} +t^2\|\Delta \varphi\|_{L^2}
  \)\lesssim \<t\>^{-\mu}.
\end{equation*}
The right hand side is integrable since $\mu>1$, so the strong limits
\begin{equation*}
   \lim_{t\to \pm\infty} U_Q(0,t)U(t)
\end{equation*}
exist in $L^2$. 
To infer that these strong limits actually exist in $\Sigma$, we
simply invoke \eqref{eq:Qhomo} in the case $(q_1,r_1)=(\infty,2)$, so
the above computation are easily adapted. 
\smallbreak

For asymptotic completeness, we can adopt the same strategy. Indeed,
it suffices to prove that  for all $\varphi\in 
  \Sch(\R^d)$, 
\begin{equation*}
    t\mapsto \left\| U(-t) \<Q(t)y,y\>U_Q(t,0)\varphi\right\|_{\Sigma}\in
    L^1(\R). 
  \end{equation*}
For the $L^2$ norm, we have
\begin{align*}
  \left\| U(-t) \<Q(t)y,y\>U_Q(t,0)\varphi\right\|_{L^2}&= \left\|
   \<Q(t)y,y\>U_Q(t,0)\varphi\right\|_{L^2}\\
& \lesssim
 \<t\>^{-\mu-2}\sum_{j=1}^d  \left\|
   y_j^2U_Q(t,0)\varphi\right\|_{L^2}.
\end{align*}
We first proceed like above, and write
\begin{equation*}
  y_j^2 =  (y_j+it\d_j)^2
  -t^2\d_j^2 -2it(y_j+it\d_j)\d_j.
\end{equation*}
The operator $J$ does not commute with $U_Q$, but this lack of
commutation is harmless for our present goal, from
\eqref{eq:Qhomo}. By considering the system satisfied by
$$(y_j+it\d_j)^2U_Q(t,0)\varphi, \d_j^2 U_Q(t,0)\varphi,
\d_j(y_j+it\d_j)U_Q(t,0)\varphi,$$ 
we obtain 
\begin{align*}
 \sum_{j=1}^d&\( \| (y_j+it\d_j)^2U_Q(t,0)\varphi\|_{L^2} + \|\d_j^2
 U_Q(t,0)\varphi\|_{L^2} +
 \|\d_j(y_j+it\d_j)U_Q(t,0)\varphi\|_{L^2}\)\\
&\le C \|\varphi\|_{\Sigma^2},
\end{align*}
where $\Sigma^k$ is the space of $H^k$ functions with $k$ momenta in
$L^2$, and $C$ does not depend on time. Finally, we also have a
similar estimate by considering one more derivative or momentum. The
key remark in the computation is that the external
potential $\<Q(t)y,y\>$ is exactly quadratic in space, and so
differentiating it three times with any space variables yields zero. 
\end{proof}

\section{Proof of Theorem~\ref{theo:cv}}
\label{sec:cv}

The main result of this section is:
\begin{theorem}\label{theo:cv-unif}
  Let $d=3$, $\si=1$, $V$ as in Theorem~\ref{theo:scatt-quant}, and
  $u_-\in \Sigma^7$. Suppose that Assumption~\ref{hyp:flot} 
  is satisfied. Let $\psi^\eps$ be given by
  Theorem~\ref{theo:scatt-quant}, $u$ be given by
  Theorem~\ref{theo:scatt-class}, $\varphi^\eps$ defined by
  \eqref{eq:phi}.  We have
  the uniform error 
  estimate:
  \begin{equation*}
   \sup_{t\in \R}\|\psi^\eps(t)-\varphi^\eps(t)\|_{L^2(\R^3)}  =
    \O\(\sqrt\eps\). 
  \end{equation*}
\end{theorem}
Theorem~\ref{theo:cv} is a direct consequence of the above
result, whose proof is the core of
Section~\ref{sec:cv}. From now on, we assume $d=3$ and $\si=1$. 
\subsection{Extra properties for the approximate solution}
\label{sec:extra-u}

Further regularity and localization properties on $u$ will be
needed. 
\begin{proposition}\label{prop:extra-u}
  Let $\si=1$, $1\le d\le 3$, $k\ge 2$ and $V$ satisfying
  Assumption~\ref{hyp:V} for some $\mu>1$. If $u_-\in \Sigma^k$, then
  the solution $u\in C(\R;\Sigma)$ provided by Theorem~\ref{theo:scatt-class}
  satisfies $u\in C(\R;\Sigma^k)$.  The momenta
  of $u$ satisfy
  \begin{equation*}
    \lVert \lvert y\rvert^\ell u(t,y)\|_{L^2(\R^d)}\le C_\ell
    \<t\>^\ell,\quad 0\le \ell\le k,
  \end{equation*}
where $C_\ell$ is independent of $t\in \R$.
 \end{proposition}
\begin{proof}
  We know from the proof of Theorem~\ref{theo:scatt-class} that since
  $u_-\in \Sigma$,
\begin{equation*}
  u,\nabla u, Ju \in L^\infty(\R;L^2(\R^d)).
  \end{equation*}
The natural approach is  then to proceed by induction on $k$,
to prove that 
\begin{align*}
  \nabla^k u,J^k u\in L^\infty(\R;L^2(\R^d)). 
\end{align*}
 We have, as we have seen in the proof of
Proposition~\ref{prop:wave-class},
\begin{align*}
   i\d_t \nabla u &= H_Q \nabla u + Q(t)y u +\nabla
                    \(|u|^2 u\)\\
&+ H_Q \nabla u + Q(t)J(t)u -it Q(t)\nabla u +\nabla
                    \(|u|^2 u\),\\
i\d_t Ju  & = H_Q Ju   +it Q(t)y u +J
                    \(|u|^2 u\)\\
& = H_Q J u + itQ(t)J(t)u +t^2Q(t)\nabla u +J
                    \(|u|^2 u\).
\end{align*}
Applying the operators $\nabla$ and $J$ again, we find
\begin{align*}
   i\d_t \nabla^2 u &= H_Q \nabla^2 u + 2Q(t)y \nabla u +Q(t) u +\nabla^2
                    \(|u|^2 u\)\\
&+ H_Q \nabla u + 2Q(t)J(t)\nabla u -2it Q(t)\nabla^2 u+Q(t)u +\nabla^2
                    \(|u|^2 u\),\\
i\d_t J^2u  & = H_Q J^2u   -2t^2 Q(t)y \nabla u -t^2Q(t)u+J^2
                    \(|u|^2 u\)\\
& = H_Q J^2 u - 2t^2Q(t)J\nabla u +2it^3Q(t) J^2 u +itQ(t)u+J^2
                    \(|u|^2 u\).
\end{align*}
In view of \eqref{eq:decayQ}, we see that $t\mapsto t^3 Q(t)$ need not be
integrable (unless we make stronger and stronger assumptions of $\mu$,
as $k$ increases), so the commutator seems to be fatal to this approach. To
overcome this issue, we use the vector field mentioned in
Section~\ref{sec:vector-field-Q}. 
For bounded time $t\in
[-T,T]$, the above mentioned lack of commutation is not a problem, and
we can use the operator $J$, which is defined for all time. 
We note that either of the operators $\mathcal A_\pm$ or 
$J$ satisfies more generally the pointwise identity
\begin{equation*}
  B\(u_1\overline u_2 u_3\) =\( B u_1\) \overline u_2 u_3 +
  u_1\(\overline{B u_2}\) u_3 + u_1\overline u_2\( Bu_3\),
\end{equation*}
for all differentiable functions $u_1,u_2,u_3$. 

Now we have all the tools to proceed by induction, and mimic the
proof from \cite[Appendix]{Ca11}. The main idea is that the proof is
similar to the propagation of higher regularity for energy-subcritical
problems, with the difference that large time is handled thanks to
vector fields. We leave out the details, which are not difficult but
rather cumbersome: considering
\begin{equation*}
  B(t) =
\left\{
  \begin{aligned}
    \mathcal A_-(t)&\text{ for }t\le -T,\\
J(t)&\text{ for }t\in [-T,T],\\
 \mathcal A_+(t)&\text{ for }t\ge T,
  \end{aligned}
\right.
\end{equation*}
 we can then prove that 
 \begin{equation*}
   \nabla^k u,B^k u\in L^\infty(\R;L^2(\R^d)).  
 \end{equation*}
Back to the definition of $\mathcal A_\pm$, 
\begin{equation*}
  \mathcal A_\pm (t) = W_\pm (t)K_\pm (t)y +iW_\pm (t)\nabla,
\end{equation*}
\eqref{eq:vector-asym}
then yields the result. 
\end{proof}

\subsection{Strichartz estimates}
\label{sec:strichartz-raff}

Introduce the following notations, taking the dependence upon $\eps$
into account:
\begin{equation*}
  H^\eps=-\frac{\eps^2}{2}\Delta+V(x),\quad U_V^\eps(t) =
  e^{-i\frac{t}{\eps} H^\eps}. 
\end{equation*}
Since we now work only in space dimension $d=3$, we can use the result
from \cite{Go06}. Resuming the proof from \cite{Go06} (a mere scaling
argument is not sufficient), we have, along with the preliminary
analysis from Section~\ref{sec:spectral}, the global dispersive estimate
\begin{equation}
  \label{eq:disp-semi-glob}
  \|U^\eps_V(t)\|_{L^1(\R^3)\to L^\infty(\R^3)}\lesssim
  \frac{1}{(\eps|t|)^{3/2}},\quad t\not =0.
\end{equation}
For $|t|\le \delta$, $\delta>0$ independent of $\eps$, the above
relation stems initially from \cite{Fujiwara}. As a consequence, we
can measure the dependence upon $\eps$ in Strichartz estimates. We
recall the definition of admissible pairs related to Sobolev
regularity.
\begin{definition}
  Let $d=3$ and $s\in \R$. A pair $(q,r)$ is called $\dot H^s$-admissible if 
  \begin{equation*}
    \frac{2}{q}+\frac{3}{r} = \frac{3}{2}-s. 
  \end{equation*}
\end{definition}
For $t_0\in \R\cup \{-\infty\}$, we denote by
\begin{equation*}
  R^\eps_{t_0}(F)(t) = \int_{t_0}^t U_V^\eps(t-s)F(s)ds
\end{equation*}
the retarded term related to Duhamel's formula. Since the dispersive
estimate \eqref{eq:disp-semi-glob} is the same as the one for
$e^{i\eps t\Delta}$, we get the same scaled Strichartz estimates as
for this operator, which can in turn be obtained by scaling
arguments from the case $\eps=1$. 
\begin{lemma}[Scaled $L^2$-Strichartz estimates]\label{lem:stri-eps}
  Let $t_0\in \R\cup\{-\infty\}$, and let $(q_1,r_1)$ and $(q_2,r_2)$
  be $L^2$-admissible pairs, $2\le  r_j\le 6$. We have
  \begin{equation*}
   \eps^{\frac{1}{q_1}} \|U_V^\eps(\cdot) f\|_{L^{q_1}(\R;L^{r_1}(\R^3))}\lesssim
    \|f\|_{L^2(\R^3)}, 
  \end{equation*}
  \begin{equation*}
  \eps^{\frac{1}{q_1}+\frac{1}{q_2}}
  \|R^\eps_{t_0}(F)\|_{L^{q_1}(I;L^{r_1}(\R^3))}\le C_{q_1,q_2} \|F\|_{L^{q_2'}(I;L^{r_2'}(\R^3))},
  \end{equation*}
where $C_{q_1,q_2}$ is independent of $\eps$, $t_0$, and of $I$ such that
$t_0\in \bar I$. 
\end{lemma}
We will also use Strichartz estimates for non-admissible pairs, as
established in \cite{Kat94} (see also \cite{CW92,FoschiStri}).
\begin{lemma}[Scaled inhomogeneous Strichartz
  estimates]\label{lem:stri-inhom-eps} 
   Let $t_0\in \R\cup\{-\infty\}$, and let $(q_1,r_1)$ be an $\dot
   H^{1/2}$-admissible pair, and $(q_2,r_2)$
  be an $\dot H^{-1/2}$-admissible pair, with 
  \begin{equation*}
    3\le r_1,r_2<6.
  \end{equation*}
We have
\begin{equation*}
  \eps^{\frac{1}{q_1}+\frac{1}{q_2}}
  \|R^\eps_{t_0}(F)\|_{L^{q_1}(I;L^{r_1}(\R^3))}\le C_{q_1,q_2} \|F\|_{L^{q_2'}(I;L^{r_2'}(\R^3))},
  \end{equation*}
where $C_{q_1,q_2}$ is independent of $\eps$, $t_0$, and of $I$ such that
$t_0\in \bar I$. 
\end{lemma}
\subsection{Preparing the proof}
\label{sec:preparing-proof}

 Subtracting the equations satisfied by
$\psi^\eps$ and $\varphi^\eps$, respectively, we obtain as in
\cite{CaFe11}: $w^\eps=\psi^\eps-\varphi^\eps$ satisfies
\begin{equation}\label{eq:restecrit}
  i\eps\d_t w^\eps +\frac{\eps^2}{2}\Delta w^\eps =V w^\eps -\mathcal L^\eps
  + 
  \eps^{5/2}\(|\psi^\eps|^{2}\psi^\eps
  -|\varphi^\eps|^{2}\varphi^\eps\),
\end{equation}
along with the initial condition
\begin{equation*}
    e^{-i\frac{\eps
      t}{2}\Delta}w^\eps_{\mid t=-\infty}=0,  
\end{equation*}
where the source term is given by 
\begin{equation*}
  {\mathcal L}^\eps(t,x) = \(V(x) - V\(q(t)\)
  -\sqrt\eps \<\nabla V\(q(t)\),y\>
  -\frac{\eps}{2}\<Q(t)y,y\>\)\Big|_{y=\frac{x-q(t)}{\sqrt\eps}}
  \varphi^\eps(t,x). 
\end{equation*}
Duhamel's
formula for $w^\eps$ reads
\begin{align*}
  w^\eps(t) &= -i\eps^{3/2}\int_{-\infty}^t U^\eps_V(t-s)\(|\psi^\eps|^{2}\psi^\eps
  -|\varphi^\eps|^{2}\varphi^\eps\)(s)ds\\
&\quad  +i\eps^{-1}\int_{-\infty}^t U^\eps_V(t-s) \mathcal L^\eps(s)ds. 
\end{align*}
Denoting $L^a(]-\infty,t];L^b(\R^3))$ by $L^a_tL^b$, Strichartz
estimates yield, for any $L^2$-admissible pair $(q_1,r_1)$,
\begin{equation}\label{eq:stri-weps}
  \eps^{1/q_1}\|w^\eps\|_{L^{q_1}_t L^{r_1}} \lesssim
  \eps^{3/2-1/q}\left\||\psi^\eps|^{2}\psi^\eps 
  -|\varphi^\eps|^{2}\varphi^\eps\right\|_{L^{q'}_tL^{r'}} +
  \frac{1}{\eps}\|\mathcal L^\eps\|_{L^1_tL^2},
\end{equation}
where $(q,r)$ is the admissible pair chosen in the proof of
Proposition~\ref{prop:waveop-quant}, that is $r=2\si+2$. Since we now
have  $d=3$ and $\si=1$, this means:
\begin{equation*}
  q=\frac{8}{3},\quad k=8,
\end{equation*}
and \eqref{eq:stri-weps} yields
\begin{equation}\label{eq:w-presque}
  \eps^{1/q_1}\|w^\eps\|_{L^{q_1}_t L^{r_1}} \lesssim
  \eps^{9/8}\( \|w^\eps\|^2_{L^8_t L^4}+ \|\varphi^\eps\|^2_{L^8_t
    L^4}\)\|w^\eps\|_{L^{8/3}_tL^4}  +
  \frac{1}{\eps}\|\mathcal L^\eps\|_{L^1_tL^2}.
\end{equation}
The strategy is then to first
obtain an a priori estimate for $w^\eps$ in $L^8_tL^4$, and then to
use it in the above estimate. In order to do so, we begin by
estimating the source term $\mathcal L^\eps$, in the next subsection. 
\subsection{Estimating the source term}
\label{sec:estim-source-term}

\begin{proposition}\label{prop:est-source}
  Let $d= 3$, $\si=1$, $V$ satisfying Assumption~\ref{hyp:V}
  with $\mu>2$, and $u_-\in \Sigma^k$ for some $k\ge
  7$. Suppose that Assumption~\ref{hyp:flot} is satisfied.
Let $u\in C(\R;\Sigma^k)$ given by
  Theorem~\ref{theo:scatt-class} and 
  Proposition~\ref{prop:extra-u}. The source term $\mathcal L^\eps$ satisfies
  \begin{equation*}
   \frac{1}{\eps} \|\mathcal L^\eps(t)\|_{L^2(\R^3)}\lesssim \frac{\sqrt
     \eps}{\<t\>^{3/2} }\quad\text{and}\quad  \frac{1}{\eps}
   \|\mathcal L^\eps(t)\|_{L^{3/2}(\R^3)}\lesssim \frac{
     \eps^{3/4}}{\<t\>^{3/2} },
\quad \forall
    t\in \R.
  \end{equation*}
\end{proposition}
\begin{proof}
To ease notation, we note that
\begin{equation*}
 \frac{1}{\eps} \mathcal L^\eps(t,x) = \frac{1}{\eps^{3/4}} {\mathcal
    S}^\eps(t,y)\Big|_{y=\frac{x-q(t)}{\sqrt\eps}}
  e^{i(S(t)+ip(t)\cdot (x-q(t)))/\eps}, 
\end{equation*}
where
\begin{equation*}
  {\mathcal S}^\eps(t,y) = \frac{1}{\eps}\( V\(q(t)+y\sqrt \eps\) -V\(q(t)\)
  -\sqrt\eps \<\nabla V\(q(t)\),y\> -\frac{\eps}{2}\<Q(t)y,y\>\)u(t,y).
\end{equation*}
In particular,
\begin{equation*}
  \frac{1}{\eps}\| \mathcal L^\eps(t)\|_{L^2(\R^3)} = \|{\mathcal
    S}^\eps(t)\|_{L^2(\R^3)} ,\quad  \frac{1}{\eps}\| \mathcal
  L^\eps(t)\|_{L^{3/2}(\R^3)} = \eps^{1/4}\|{\mathcal 
    S}^\eps(t)\|_{L^{3/2}(\R^3)}. 
\end{equation*}
  Taylor's formula and Assumption~\ref{hyp:V} yield the pointwise estimate
  \begin{equation*}
    | {\mathcal S}^\eps(t,y) | \lesssim \sqrt\eps |y|^3 \int_0^1
    \frac{1}{\<q(t)+\theta y\sqrt \eps\>^{\mu+3}}d\theta |u(t,y)|. 
  \end{equation*}
To simplify notations, we consider only positive times. Recall that
from Assumption~\ref{hyp:flot}, $p^+\not =0$. Introduce, for
$0<\eta< |p^+|/2$,
\begin{equation*}
  \Omega  = \left\{y\in \R^3,\quad |y|\ge \eta\frac{t}{\sqrt\eps}\right\}.
\end{equation*}
Since $q(t)\sim p^+ t$ as
$t\to \infty$, on the complement of $\Omega$, we can use the decay of $V$,
\eqref{eq:short}, to infer the pointwise estimate
\begin{equation}\label{eq:Spoint}
   | {\mathcal S}^\eps(t,y) | \lesssim \sqrt\eps |y|^3
    \frac{1}{\<t\>^{\mu+3}} |u(t,y)| \quad \text{on }\Omega^c.
\end{equation}
Taking the $L^2$-norm, we have
\begin{equation*}
  \|\mathcal S^\eps(t)\|_{L^2(\Omega^c)}\le \frac{\sqrt \eps
  }{\<t\>^{\mu+3}}\||y|^3u(t,y)\|_{L^2(\R^3)}\lesssim \frac{\sqrt \eps
  }{\<t\>^{\mu}},
\end{equation*}
where we have used Proposition~\ref{prop:extra-u}. On $\Omega$
however, the argument of the potential in Taylor's formula is not
necessarily going to infinity, so the decay of the potential is
apparently useless. Back to the definition of $\mathcal L^\eps$, that is leaving
out Taylor's formula, we see that all the terms but the first one can
be easily estimated on $\Omega$. Indeed, the definition of $\Omega$ implies
\begin{equation*}
  |V(q(t))u(t,y)| \lesssim \frac{1}{\<t\>^\mu}|u(t,y)|\lesssim
  \frac{1}{\<t\>^\mu} \left| \frac{y\sqrt \eps}{t}\right|^k |u(t,y)|,
\end{equation*}
where $k$ will be chosen shortly. Taking the $L^2$ norm, we find
\begin{equation*}
  \frac{1}{\eps}\|V(q(t))u(t)\|_{L^2(\Omega)} \lesssim
  \frac{\eps^{k/2-1}}{\<t\>^{\mu+k}}\||y|^k u(t,y)\|_{L^2(\R^3)}
  \lesssim \frac{\eps^{k/2-1}}{\<t\>^{\mu}},
\end{equation*}
where we have used Proposition~\ref{prop:extra-u} again. Choosing
$k=3$ yields the expected estimate. The last two terms in $\mathcal
L^\eps$ can be estimated accordingly. For the first term in $\mathcal
L^\eps$ however, we face the same problem as above: the argument of
$V$ has to be considered as bounded. A heuristic argument goes as
follows. In view of Theorem~\ref{theo:scatt-class},
\begin{equation*}
  u(t,y) \Eq t \infty e^{i\frac{t}{2}\Delta}u_+ \Eq t \infty
  \frac{1}{t^{3/2}}\widehat u_+\(\frac{y}{t}\)e^{i|y|^2/(2t)},
\end{equation*}
where the last behavior stems from standard analysis of the
Schr\"odinger group (see e.g. \cite{Rauch91}). In view of
the definition of $\Omega$, we have, formally for $y\in \Omega$,  
\begin{equation*}
  |u(t,y)|\lesssim
  \frac{1}{t^{3/2}}\sup_{ |z|\ge \eta}\left |\widehat
    u_+\(\frac{z}{\sqrt\eps}\)\right|. 
\end{equation*}
Then the idea is to keep the linear dispersion measured by the factor
$t^{-3/2}$ (which is integrable since $d=3$), and use decay properties
for $\widehat u_+$ to gain powers of $\eps$. To make this argument
rigorous, we keep the idea 
that $u$ must be assessed in $L^\infty$ rather than in $L^2$, and write
\begin{equation*}
  \frac{1}{\eps}\|V\(q(t)+y\sqrt \eps\)u(t,y)\|_{L^2(\Omega)} \le
  \frac{1}{\eps}\|u(t)\|_{L^\infty(\Omega)} \|V\(q(t)+y\sqrt \eps\)\|_{L^2(\Omega)} .
\end{equation*}
For the last factor, we have
\begin{equation*}
  \|V\(q(t)+y\sqrt
  \eps\)\|_{L^2(\Omega)}\le \eps^{-3/4}\|V\|_{L^2(\R^3)},
\end{equation*}
where the last norm is finite since $\mu>2$. For the $L^\infty$ norm of
$u$, we use Gagliardo-Nirenberg inequality and the previous
vector-fields. To take advantage of the localization in space,
introduce a non-negative cut-off function $\chi\in C^\infty(\R^3)$, such that:
\begin{equation*}
  \chi(z)=\left\{
    \begin{aligned}
      1& \text{ if }|z|\ge \eta,\\
0 & \text{ if }|z|\le\frac{\eta}{2}.
    \end{aligned}
\right.
\end{equation*}
In view of the definition of $\Omega$, 
\begin{equation*}
   \|u(t)\|_{L^\infty(\Omega)}  \le \left\|
     \chi\(\frac{y\sqrt\eps}{t}\)u(t,y)\right\|_{L^\infty(\R^3)}. 
\end{equation*}
Now with $B$ as defined in the proof of
Proposition~\ref{prop:extra-u}, Gagliardo-Nirenberg inequality yields,
for any smooth function $f$
 (recall that $y\in \R^3$),
 \begin{equation*}
   \|f\|_{L^\infty(\R^3)}\lesssim
   \frac{1}{t^{3/2}}\|f\|_{L^2(\R^3)}^{1/4}\|B^2(t)f\|_{L^2(\R^3)}^{3/4}.
 \end{equation*}
We use this inequality with
\begin{equation*}
  f(t,y) = \chi\(\frac{y\sqrt\eps}{t}\)u(t,y),
\end{equation*}
and note that
\begin{equation*}
  B(t)f (t,y)= \chi\(\frac{y\sqrt\eps}{t}\)B(t)u(t,y) + i\frac{\sqrt
    \eps}{t}W(t) \nabla \chi \(\frac{y\sqrt\eps}{t}\) \times u(t,y),
\end{equation*}
where $W(t)$ stands for $W_\pm$ or $t$. Recall that $t\mapsto W(t)/t$
is bounded, so the last term is actually ``nice''. 
Proceeding in the same way as above, we obtain
\begin{equation*}
  \|u(t)\|_{L^2(\Omega)}\lesssim
  \left\| \left| \frac{y\sqrt \eps}{t}\right|^k
    u(t,y)\right\|_{L^2(\Omega)}\lesssim \eps^{k/2},
\end{equation*}
provided that $u_-\in \Sigma^k$. Similarly,
\begin{equation*}
  \|B^2(t)u\|_{L^2(\Omega)} \lesssim \eps^{k/2-1},
\end{equation*}
and so 
\begin{equation*}
   \frac{1}{\eps}\|V\(q(t)+y\sqrt \eps\)u(t,y)\|_{L^2(\Omega)}\lesssim 
   \frac{1}{t^{3/2}}\eps^{-7/4 +k/8+3(k/2-1)/4}= \frac{\eps^{k/2-5/2}}{t^{3/2}}.
\end{equation*}
Therefore, the $L^2$ estimate follows as soon as $k \ge 6$. For the
$L^{3/2}$-estimate, we resume the same computations, and use the extra
estimate: for all $s>1/2$, 
\begin{equation}\label{eq:localizing}
  \|f\|_{L^{3/2}(\R^3)}\lesssim \|f\|_{L^2(\R^3)}^{1-1/2s}\||x|^sf\|_{L^2(\R^3)}^{1/2s}.
\end{equation}
This estimate can easily be proven by writing
\begin{equation*}
   \|f\|_{L^{3/2}(\R^3)} \le \|f\|_{L^{3/2}(|y|<R)} + \left\|
     \frac{1}{|x|^s} |x|^s f\right\|_{L^{3/2}(|x|>R)},
\end{equation*}
so H\"older inequality yields, provided that $s>1/2$ (so that
$y\mapsto |y|^{-s}\in L^6(|y|>R)$)
\begin{equation*}
   \|f\|_{L^{3/2}(\R^3)} \le \sqrt R \|f\|_{L^2} + \frac{1}{R^{s-1/2}}\||x|^s f\|_{L^2},
\end{equation*}
and by optimizing in $R$. Now from \eqref{eq:Spoint}, we have
\begin{align*}
  \|\mathcal S^\eps(t)\|_{L^{3/2}(\Omega^c)}&\le \frac{\sqrt \eps
  }{\<t\>^{\mu+3}}\||y|^3u(t,y)\|_{L^{3/2}(\R^d)}\\
&\lesssim \frac{\sqrt \eps
  }{\<t\>^{\mu+3}}\||y|^3u(t,y)\|_{L^2(\R^d)}^{1/2}\||y|^4u(t,y)\|_{L^2(\R^d)}^{1/2}\\
&
\lesssim \frac{\sqrt \eps
  }{\<t\>^{\mu-1/2}}\lesssim \frac{\sqrt \eps
  }{\<t\>^{3/2}}
\end{align*}
where we have used \eqref{eq:localizing} with $s=1$, 
Proposition~\ref{prop:extra-u}, and the fact that $\mu>2$. 

On $\Omega$, we can repeat the computations from the $L^2$-estimate
(up to incorporating \eqref{eq:localizing}): for the last term, we
note that
\begin{equation*}
  \frac{1}{\eps}\|V\(q(t)+y\sqrt \eps\)u(t,y)\|_{L^{3/2}(\Omega)} \le
  \frac{1}{\eps}\|u(t)\|_{L^\infty(\Omega)} \|V\(q(t)+y\sqrt \eps\)\|_{L^{3/2}(\Omega)} ,
\end{equation*}
and that
\begin{equation*}
  \|V\(q(t)+y\sqrt
  \eps\)\|_{L^{3/2}(\Omega)}\le \eps^{-1}\|V\|_{L^{3/2}(\R^3)},
\end{equation*}
where the last norm is finite since $\mu>2$. Up to taking $u$ in
$\Sigma^7$, we conclude
\begin{equation*}
   \|\mathcal S^\eps(t)\|_{L^{3/2}(\R^3)}\lesssim \frac{\sqrt\eps}{\<t\>^{3/2}},
\end{equation*}
and the proposition follows. 
\end{proof}

\subsection{A priori estimate for the error in the critical norm}
\label{sec:priori-estim-error}

In this subsection, we prove:
\begin{proposition}\label{prop:w-crit}
  Under the assumptions of Theorem~\ref{theo:cv-unif}, the error
  $w^\eps=\psi^\eps-\varphi^\eps$  satisfies the a priori estimate,
  for any $\dot H^{1/2}$-admissible pair 
  $(q,r)$, 
  \begin{equation*}
    \eps^{\frac{1}{q}}\|w^\eps\|_{L^q(\R;L^r(\R^3))}\lesssim \eps^{1/4}. 
  \end{equation*}
\end{proposition}
\begin{proof}
  The reason for considering $\dot H^{1/2}$-admissible pairs is that
  the cubic three-dimensional Schr\"odinger equation is $\dot
  H^{1/2}$-critical; see e.g. \cite{CW90}. The proof of
  Proposition~\ref{prop:w-crit} is then very similar to the proof of
  \cite[Proposition~2.3]{HoRo08}. 

An important tool is the known estimate for the approximate solution
$\varphi^\eps$: we have, in view of the fact
that $u,Bu\in L^\infty L^2$,
\begin{equation}\label{eq:est-a-priori-phi}
  \|\varphi^\eps(t)\|_{L^r(\R^3)}\lesssim
  \(\frac{1}{\<t\>\sqrt\eps}\)^{3\(\frac{1}{2}-\frac{1}{r}\)},\quad
  2\le r\le 6.
\end{equation}
Note that for an $\dot H^{1/2}$ admissible pair, we infer
\begin{equation*}
  \|\varphi^\eps(t)\|_{L^q(\R;L^r(\R^3))}\lesssim
  \eps^{-\frac{3}{2}\(\frac{1}{2}-\frac{1}{r}\)} = \eps^{-\frac{1}{q}-\frac{1}{4}},
\end{equation*}
so Proposition~\ref{prop:w-crit} shows a $\sqrt\eps$ gain
for $w^\eps$ compared to $\varphi^\eps$, which is the order of
magnitude we eventually prove in $L^\infty L^2$, and stated in
Theorem~\ref{theo:cv-unif}.  
Let $0<\eta\ll 1$, and set
\begin{equation*}
  \|w^\eps\|_{\mathcal N^\eps(I)} :=\sup_{{(q,r)\ \dot
      H^{1/2}-\text{admissible}}\atop 3\le r\le
    6-\eta}\eps^{\frac{1}{q}}\|w^\eps\|_{L^q(I;L^r(\R^3)}. 
\end{equation*}
Duhamel's formula for \eqref{eq:restecrit} reads, given $w^\eps_{\mid
  t=-\infty}=0$,
\begin{equation*}
  w^\eps(t) =-i\eps^{3/2} \int_{-\infty}^t
  U_V^\eps(t-s)\(|\psi^\eps|^2\psi^2-|\varphi^\eps|^2\varphi^\eps\)(s)ds
+i\eps^{-1} \int_{-\infty}^t  U_V^\eps(t-s)\mathcal L^\eps(s)ds. 
\end{equation*}
Since we have the point-wise estimate
\begin{equation*}
  \left|
    |\psi^\eps|^2\psi^2-|\varphi^\eps|^2\varphi^\eps\right|\lesssim
  \(|w^\eps|^2+|\varphi^\eps|^2\)|w^\eps|, 
\end{equation*}
Lemma~\ref{lem:stri-inhom-eps} yields, with
$(q_2,r_2)=(\frac{10}{7},5)$ for the first term of the right hand
side, and with $(q_2,r_2)=(2,3)$ for the second term,
\begin{align*}
  \|w^\eps\|_{\mathcal N^\eps(-\infty,t)}&\lesssim \eps^{3/2-7/10}\left\|
    \(|w^\eps|^2+|\varphi^\eps|^2\)w^\eps\right\|_{L^{10/3}_tL^{5/4}}
  + \eps^{-3/2}\|\mathcal L^\eps\|_{L^2_t L^{3/2}}\\
&\lesssim \eps^{4/5}\( \|w^\eps\|_{L^{20}_tL^{10/3}}^2 +
\|\varphi^\eps\|_{L^{20}_tL^{10/3}}^2 \)
\|w^\eps\|_{L^{5}_tL^{5}}
  + \eps^{-3/2}\|\mathcal L^\eps\|_{L^2_t L^{3/2}},
\end{align*}
where we have used H\"older inequality. Note that the pairs
$(20,\frac{10}{3})$ and $(5,5)$ are $\dot H^{1/2}$-admissible. Denote
by 
\begin{equation*}
  \omega(t) =\frac{1}{\<t\>^{3/5}}. 
\end{equation*}
This function obviously belongs to $L^{20}(\R)$. 
The estimate \eqref{eq:est-a-priori-phi} and the definition of the
norm $\mathcal N^\eps$ yield
\begin{equation*}
   \|w^\eps\|_{\mathcal N^\eps(-\infty,t)}\lesssim \sqrt\eps
   \|w^\eps\|_{\mathcal N^\eps(-\infty,t)}^3 + 
\|\omega\|_{L^{20}(-\infty,t)}^2  \|w^\eps\|_{\mathcal N^\eps(-\infty,t)}
  + \eps^{-3/2}\|\mathcal L^\eps\|_{L^2_t L^{3/2}}.
\end{equation*}
Taking $t\ll -1$, we infer
\begin{equation*}
   \|w^\eps\|_{\mathcal N^\eps(-\infty,t)}\lesssim \sqrt\eps
   \|w^\eps\|_{\mathcal N^\eps(-\infty,t)}^3 
  + \eps^{-3/2}\|\mathcal L^\eps\|_{L^2_t L^{3/2}}\lesssim \sqrt\eps
   \|w^\eps\|_{\mathcal N^\eps(-\infty,t)}^3  + \eps^{1/4},
\end{equation*}
where we have use Proposition~\ref{prop:est-source}. We can now use a
standard bootstrap argument, as recalled in Section~\ref{sec:class}.
We infer that for $t_1\ll -1$,
\begin{equation*}
    \|w^\eps\|_{\mathcal N^\eps(-\infty,t_1)}\lesssim \eps^{1/4}.
\end{equation*}
Using Duhamel's formula again, we have
\begin{align*}
    U_V^\eps(t-t_1)w^\eps(t_1) &=-i\eps^{3/2} \int_{-\infty}^{t_1}
  U_V^\eps(t-s)\(|\psi^\eps|^2\psi^2-|\varphi^\eps|^2\varphi^\eps\)(s)ds\\
&+i\eps^{-1} \int_{-\infty}^{t_1}  U_V^\eps(t-s)\mathcal L^\eps(s)ds,
\end{align*}
so we infer
\begin{align*}
    \| U_V^\eps(t-t_1)w^\eps(t_1)\|_{\mathcal N^\eps(\R)}& \lesssim  \sqrt\eps
   \|w^\eps\|_{\mathcal N^\eps(-\infty,t_1)}^3 + 
\|\omega\|_{L^{20}(-\infty,t_1)}^2  \|w^\eps\|_{\mathcal
                                                           N^\eps(-\infty,t_1)}\\
& \quad  + \eps^{-3/2}\|\mathcal L^\eps\|_{L^2((-\infty,t_1]; L^{3/2})}\\
&\le
C_0\eps^{1/4}.
\end{align*}
We now rewrite Duhamel's formula with some initial time $t_j$:
\begin{align*}
  w^\eps(t) &= U_V^\eps(t-t_j)w^\eps(t_j)-i\eps^{3/2} \int_{t_j}^t
  U_V^\eps(t-s)\(|\psi^\eps|^2\psi^2-|\varphi^\eps|^2\varphi^\eps\)(s)ds\\
&\quad
+i\eps^{-1} \int_{t_j}^t  U_V^\eps(t-s)\mathcal L^\eps(s)ds. 
\end{align*}
For $t\ge t_j$ and $I=[t_j,t]$, the same estimates as above yield
\begin{align*}
   \|w^\eps\|_{\mathcal N^\eps(I)}&\le  \|U_V^\eps(\cdot-t_j)w^\eps(t_j)\|_{\mathcal N^\eps(I)}
  + C\sqrt\eps
   \|w^\eps\|_{\mathcal N^\eps(I)}^3 + 
C\|\omega\|_{L^{20}(I)}^2  \|w^\eps\|_{\mathcal N^\eps(I)}\\
&  \quad+ C\eps^{-3/2}\|\mathcal L^\eps\|_{L^2(I; L^{3/2})},
\end{align*}
where the above constant $C$ is independent of $\eps, t_j$ and $t$. We
split $\R_t$ into finitely many intervals
\begin{equation*}
  \R = (-\infty,t_1]\cup \bigcup_{j=1}^N [t_j,t_{j+1}]\cup
[t_N,\infty)=:\bigcup_{j=0}^{N+1} I_j, 
\end{equation*}
 on which 
\begin{equation*}
  C\|\omega\|_{L^{20}(I_j)}^2 \le\frac{1}{2},
\end{equation*}
so that we have
\begin{align*}
   \|w^\eps\|_{\mathcal N^\eps(I_j)}&\le   2\|U_V^\eps(\cdot-t_j)w^\eps(t_j)\|_{\mathcal N^\eps(I_j)}
  +2 C\sqrt\eps
   \|w^\eps\|_{\mathcal N^\eps(I_j)}^3  +
  2 C\eps^{-3/2}\|\mathcal L^\eps\|_{L^2(I_j; L^{3/2})}\\
&\le  2\|U_V^\eps(\cdot-t_j)w^\eps(t_j)\|_{\mathcal N^\eps(I_j)}
  +2 C\sqrt\eps
   \|w^\eps\|_{\mathcal N^\eps(I_j)}^3  + \tilde C
  \eps^{1/4}\left\|\<t\>^{-3/2}\right\|_{L^2(I_j)}, 
\end{align*}
where we have used Proposition~\ref{prop:est-source} again. Since we
have
\begin{equation*}
   \| U_V^\eps(t-t_1)w^\eps(t_1)\|_{\mathcal N^\eps(\R)}\le C_0\eps^{1/4},
\end{equation*}
the bootstrap argument shows that at least for $\eps\le \eps_1$
($\eps_1>0$),
\begin{equation*}
    \|w^\eps\|_{\mathcal N^\eps(I_1)}\le 3
    \|U_V^\eps(\cdot-t_1)w^\eps(t_1)\|_{\mathcal N^\eps(I_1)} + \frac{3}{2} \tilde C
  \eps^{1/4}\left\|\<t\>^{-3/2}\right\|_{L^2(I_1)}.
\end{equation*}
On the other hand, Duhamel's formula implies
\begin{align*}
  U_V^\eps(t-t_{j+1})w^\eps(t_{j+1}) &=
                                       U_V^\eps(t-t_j)w^\eps(t_j)
+i\eps^{-1} \int_{t_j}^{t_{j+1}}  U_V^\eps(t-s)\mathcal L^\eps(s)ds\\
&\quad -i\eps^{3/2}
                                       \int_{t_j}^{t_{j+1}} 
  U_V^\eps(t-s)\(|\psi^\eps|^2\psi^2-|\varphi^\eps|^2\varphi^\eps\)(s)ds
. 
\end{align*}
Therefore, we infer
\begin{align*}
    \| U_V^\eps(t-t_{j+1})w^\eps(t_{j+1})\|_{\mathcal N^\eps(\R)}&\le
    \|U_V^\eps(t-t_{j})w^\eps(t_{j})\|_{\mathcal N^\eps(\R)}+ 
  + C\sqrt\eps
   \|w^\eps\|_{\mathcal N^\eps(I_j)}^3 \\
&  \quad + 
C\|\omega\|_{L^{20}(I_j)}^2  \|w^\eps\|_{\mathcal N^\eps(I_j)}
+ C\eps^{-3/2}\|\mathcal L^\eps\|_{L^2(I_j; L^{3/2})}.
\end{align*}
By induction (carrying over finitely many steps), we conclude
\begin{equation*}
   \|U_V^\eps(t-t_{j})w^\eps(t_{j})\|_{\mathcal N^\eps(\R)}
   =\O\(\eps^{1/4}\),\quad 0\le j\le N+1,
\end{equation*}
and $\|w^\eps\|_{\mathcal N^\eps(\R)}=\O\(\eps^{1/4}\)$ as announced. 
\end{proof}

\subsection{End of the argument}
\label{sec:end-argument}

Resume the estimate \eqref{eq:w-presque} with the $L^2$-admissible
pair $(q_1,r_1)= (\frac{8}{3},4)$:
\begin{equation*}
   \eps^{3/8}\|w^\eps\|_{L^{8/3}_t L^{4}} \lesssim
  \eps^{3/4}\( \|w^\eps\|^2_{L^8_t L^4}+ \|\varphi^\eps\|^2_{L^8_t
    L^4}\) \eps^{3/8}\|w^\eps\|_{L^{8/3}_tL^4}  +
  \frac{1}{\eps}\|\mathcal L^\eps\|_{L^1_tL^2}.
\end{equation*}
From Proposition~\ref{prop:w-crit} (the pair $(8,4)$ is $\dot H^{1/2}$-admissible),
\begin{equation*}
  \|w^\eps\|_{L^8(\R; L^4)}\lesssim \eps^{1/8},
\end{equation*}
and we have seen in the course of the proof that
\begin{equation*}
  \|\varphi^\eps\|_{L^8(\R; L^4)}\lesssim \eps^{-3/8}.
\end{equation*}
Therefore, we can split $\R_t$ into finitely many intervals, in a way
which is independent of $\eps $, so that 
\begin{equation*}
   \eps^{3/4}\( \|w^\eps\|^2_{L^8(I; L^4)}+ \|\varphi^\eps\|^2_{L^8(I;
    L^4)}\)\le \eta 
\end{equation*}
on each of these intervals, with $\eta$ so small that we infer
\begin{equation*}
   \eps^{3/8}\|w^\eps\|_{L^{8/3}(\R; L^{4})} \lesssim
  \frac{1}{\eps}\|\mathcal L^\eps\|_{L^1(\R;L^2)}\lesssim \sqrt\eps,
\end{equation*}
where we have used Proposition~\ref{prop:est-source}. Plugging this
estimate into  \eqref{eq:w-presque} and now taking $(q_1,r_1)$,
Theorem~\ref{theo:cv-unif} follows.

\section{Superposition}
\label{sec:superp}

In this section, we sketch the proof of
Corollary~\ref{cor:decoupling}. This result heavily relies on the
(finite time) superposition 
  principle established in \cite{CaFe11}, in the case of two initial
  coherent states with different centers in phase space. We present
  the argument in the case of two initial wave packets, and explain
  why it can be generalized to any finite number of initial coherent
  states. 
\smallbreak

Following the proof of
\cite[Proposition~1.14]{CaFe11}, we introduce the approximate
evolution of each individual initial wave packet:
\begin{equation*}
   \varphi_j^\eps(t,x)=\eps^{-3/4} u_j
\left(t,\frac{x-q_j(t)}{\sqrt\eps}\right)e^{i\left(S_j(t)+p_j(t)\cdot
    (x-q_j(t))\right)/\eps},
\end{equation*}
where $u_j$ solves \eqref{eq:u} with initial datum $a_j$. In the proof
of \cite[Proposition~1.14]{CaFe11}, the main remark is that all that
is needed is the control of a new source term, corresponding to the
interactions of the approximate solutions. Set 
\begin{equation*}
  w^\eps = \psi^\eps -\varphi_1^\eps-\varphi_2^\eps.
\end{equation*}
It solves 
\begin{equation*}
   i\eps\d_t w^\eps +\frac{\eps^2}{2}\Delta w^\eps = Vw^\eps -\mathcal
  L^\eps+\mathcal N_I^\eps+ \mathcal N_s^\eps\quad ;\quad w^\eps_{\mid t=0}=0,
\end{equation*}
where the linear source term is the same as in Section~\ref{sec:cv}
(except than now we consider the sums of two such terms), $\mathcal
N_s^\eps$ is the semilinear term
\begin{equation*}
  \mathcal N_s^\eps =\eps^{5/2}\( |w^\eps+
  \varphi_1^\eps+\varphi_2^\eps|^2 (w^\eps+
  \varphi_1^\eps+\varphi_2^\eps) - |
  \varphi_1^\eps+\varphi_2^\eps|^2 (
  \varphi_1^\eps+\varphi_2^\eps)\),
\end{equation*}
and $\mathcal N_I^\eps$ is precisely the new interaction term,
\begin{equation*}
   \mathcal N_I^\eps =\eps^{5/2}\( |
  \varphi_1^\eps+\varphi_2^\eps|^2 (
  \varphi_1^\eps+\varphi_2^\eps) - |
  \varphi_1^\eps|^2 \varphi_1^\eps-|
  \varphi_2^\eps|^2  \varphi_2^\eps\).
\end{equation*}
In \cite{CaFe11}, it is proven that if $(q_{01},p_{01})\not
=(q_{02},p_{02})$, then the possible interactions between
$\varphi_1^\eps$ and $\varphi_2^\eps$ are negligible on every finite
time interval, in the sense that
\begin{equation*}
  \frac{1}{\eps} \| \mathcal N_I^\eps\|_{L^1(0,T;L^2)}\le C(T,\gamma) \eps^\gamma,
\end{equation*}
for every $\gamma<1/2$. We infer that
$\|w^\eps\|_{L^\infty(0,T;L^2)}=\O(\eps^\gamma)$ for every $T>0$. For
$t\ge T$, we have
\begin{align*}
   \frac{1}{\eps} \| \mathcal N_I^\eps(t)\|_{L^2}& \lesssim
  \sum_{\ell_1,\ell_2\ge 1,\ \ell_1+\ell_2 =3}\left\|
  u_1^{\ell_1}\(t,y-\frac{q_1(t)-q_2(t)}{\sqrt\eps } \)
  u_2^{\ell_2}(t,y)\right\|_{L^2}\\
&\lesssim
  \sum_{\ell_1,\ell_2\ge 1,\ \ell_1+\ell_2 =3}\|
  u_1(t)\|_{L^\infty}^{\ell_1} 
\| u_2(t)\|_{L^\infty}^{\ell_2-1} 
\|  u_2(t)\|_{L^2}\lesssim \frac{1}{t^3}.
\end{align*}
Similarly, resuming the same estimates as in the proof of
Proposition~\ref{prop:est-source}, 
\begin{equation*}
   \frac{1}{\eps} \| \mathcal N_I^\eps(t)\|_{L^{3/2}}\lesssim
   \frac{\eps^{1/4}}{t^{5/2}}. 
\end{equation*}
By resuming the proof of Theorem~\ref{theo:cv-unif} on the time
interval $[T,\infty)$, we infer
\begin{equation*}
  \|w^\eps\|_{L^\infty(0,\infty;L^2)}\le C(T,\gamma)\eps^\gamma + \frac{C}{T^2}.
\end{equation*}
Therefore, 
\begin{equation*}
  \limsup_{\eps\to 0}  \|w^\eps\|_{L^\infty(0,\infty;L^2)}\lesssim \frac{1}{T^2},
\end{equation*}
for all $T>0$, hence the result by letting $T\to \infty$. 
\smallbreak

In the case of more than two initial coherent states, the idea is that
the nonlinear interaction term, $\mathcal N_I^\eps$, always contains
the product of two approximate solutions corresponding to different
trajectories in phase space. This is enough for the proof of
\cite[Proposition~1.14]{CaFe11} to go through: we always have
\begin{align*}
   \frac{1}{\eps} \| \mathcal N_I^\eps(t)\|_{L^2}& \\
\lesssim
  \sum_{{j\not =k, \ \ell_j,\ell_k\ge 1}\atop {\ell_j+\ell_k+\ell_m =3}}&\left\|
  u_j^{\ell_j}\(t,y-\frac{q_j(t)-q_k(t)}{\sqrt\eps } \)
  u_k^{\ell_k}(t,y)u_m^{\ell_m}\(t,y-\frac{q_m(t)-q_k(t)}{\sqrt\eps } \)\right\|_{L^2}\\
\lesssim
  \sum_{{j\not =k,\  \ell_j,\ell_k\ge 1}\atop {\ell_j+\ell_k+\ell_m
  =3}}&\|u_m(t)\|_{L^\infty}^{\ell_m}\left\| 
  u_j^{\ell_j}\(t,y-\frac{q_j(t)-q_k(t)}{\sqrt\eps } \)
  u_k^{\ell_k}(t,y)\right\|_{L^2},
\end{align*}
so the last factor is exactly the one considered in \cite{CaFe11} and
  above. 
\subsection*{Acknowledgements}
The author is grateful to Jean-Fran\c cois Bony, Clotilde Fermanian,
Isabelle Gallagher and Fabricio Maci\`a for fruitful discussions about
this work. 

\bibliographystyle{siam}

\bibliography{scatt}

\end{document}